\documentclass[a4paper,10pt]{article}
\pdfoutput=1

\usepackage[utf8]{inputenc}
\usepackage[english]{babel}
\usepackage[babel]{csquotes}
\usepackage{amsthm, amsmath, amssymb,amsfonts,latexsym}
\usepackage{mathtools} 
\usepackage{enumerate}
\usepackage{tikz}
\usetikzlibrary{automata,positioning}
\usepackage{fullpage}
\usepackage{listliketab}
\usepackage{bbm}
\usepackage{varioref}
\usepackage[numbers, sort]{natbib}
\usepackage{booktabs}

\numberwithin{equation}{section}

\usepackage{array}					
\newcolumntype{C}{>{\centering\arraybackslash$}p{1cm}<{$}}
\newcolumntype{D}{>{\centering\arraybackslash$}p{1.9cm}<{$}}

\newcommand{\RR}{\mathbb{R}}
\newcommand{\NN}{\mathbb{N}}
\newcommand{\Sp}{\mathcal{X}}
\newcommand{\C}{\mathcal{C}}
\newcommand{\R}{\mathcal{R}}
\newcommand{\Rr}{\mathcal{R}^r}
\newcommand{\Rl}{\mathcal{R}^\infty}
\newcommand{\Cl}{\mathcal{C}^\infty}
\newcommand{\betar}{\beta^r}
\newcommand{\lambdal}{\lambda^\infty}
\newcommand{\lambdarN}{\lambda^{r,N}}
\newcommand{\mlambdarN}{\widetilde{\lambda}^{r,N}}
\newcommand{\mlambdal}{\widetilde{\lambda}^\infty}

\newcommand{\K}{\mathcal{K}}
\newcommand{\Kl}{\mathcal{K}^\infty}
\newcommand{\pr}[1]{\left( #1 \right)}

\newcommand{\prg}[1]{\left\{ #1 \right\}}

\newcommand{\is}{H}
\newcommand{\In}{\mathcal{V}}
\newcommand{\FP}{\mathcal{W}}
\newcommand{\SR}{\mathcal{U}}
\newcommand{\fp}{W}
\newcommand{\ii}{V}

\newcommand{\sr}{U}

\newcommand{\La}{L}
\newcommand{\xin}{\check{x}}
\newcommand{\xni}{\hat{x}}
\newcommand{\vkappaj}{\kappa_{\cdot j}}
\newcommand{\rti}{\overset{\In}{\Rightarrow}}

\title{Uniform Approximation of Solutions by Elimination of Intermediate Species in Deterministic Reaction Networks}
\author{Daniele Cappelletti\footnotemark[1] \and Carsten Wiuf\footnotemark[1]}
\date{}

\newtheorem{theorem}{Theorem}[section]
\newtheorem{lemma}[theorem]{Lemma}

\newtheorem{proposition}[theorem]{Proposition}

\theoremstyle{remark}
\newtheorem{remark}{Remark}[section]
\newtheorem{example}{Example}

\theoremstyle{definition}
\newtheorem{definition}{Definition}[section]
\newtheorem{assumption}{Assumption}

\begin{document}

 \footnotetext[1]{Department of Mathematical Sciences, University of Copenhagen. The authors are supported by the Danish Research Council and the Lundbeck Foundation, Denmark.}

 \tikzset{every node/.style={auto}}
 \tikzset{every state/.style={rectangle, minimum size=0pt, draw=none, font=\normalsize}}
 \tikzset{bend angle=20}

 \maketitle

 \begin{abstract}
Chemical reactions often proceed through the formation and the consumption of intermediate species. An example is the creation and subsequent degradation of the substrate-enzyme complexes in an enzymatic reaction. In this paper we provide a setting, based on ordinary differential equations, in which the presence of intermediate species has little effect on the overall dynamics of a biological system. The result provides a method to perform model reduction by elimination of intermediate species. We study the problem in a multiscale setting, where the species abundances as well a the reaction rates scale to different orders of magnitudes. The different time and concentration scales are parameterised by a single parameter $N$.  We show that a solution to the original reaction  system is uniformly approximated on compact time intervals to a solution of a reduced reaction system without intermediates and to a solution of  a certain limiting reaction systems, which does not depend on $N$.
   Known approximation techniques such as the theorems by Tikhonov and Fenichel cannot readily be used in this framework.
 \end{abstract}

 \section{Introduction}
  
  Dynamical models of chemical reaction systems in biology and biochemistry have a history of more than one hundred years. Today such models are playing a crucial role in understanding the dynamic behaviour of biological and cellular systems, such as signalling pathways and the circadian clock. These  systems are typically large with reactions running at different time scales and species abundances  spanning several orders of magnitude.  
The  choice of model typically reflects this.  If the species are in low abundance, stochastic fluctuations should not be ignored and the preferred model is often a continuous-time Markov chain, where the variables are the molecular counts of  each species. In contrast,  if the species are in large abundance, it is custom to consider the concentrations of the  species, rather than the counts,  and the dynamics is modelled as a  system of ordinary differential equations (ODEs). In this paper, we are concerned about the latter class of models. 

In many situations there is an interest in reducing an ODE  system to a smaller system, either because the original system is mathematically and computationally intractable or because the full original system obscures the  essential biological aspects of the system \cite{cornish:enz-kinetics}. The quasi-steady state approximation and rapid equilibrium assumption are commonly used  techniques to reduce a system by making heuristic arguments about fast/slow reactions and fast/slow species (in our context it would be more correct to say high/low concentrations of species) \cite{cornish:enz-kinetics}. When applying time-scale separation, the species are typically divided into fast and slow species, and based on this division  a reduced ODE system is deduced with only the slow species.  This is often done without mathematical justification. However, Tikhonov's theorem (and similar theorems) often applies and allows us to conclude that the trajectories of the original system is uniformly approximated by the trajectories  of the reduced system on compact time intervals \cite{tikhonov,fenichel,segal-slemrod,goeke-walcher}.

  The present paper addresses at the same time, the issue of time-scale separation by means of high/low concentrations  as well as by fast/slow reactions, in the specific context of  reaction networks with intermediate species.
  Intermediate species  are  transient species in a reaction pathway, which are  produced and subsequently degraded in the course of time (a formal definition is given in Section \ref{sec:def}).  A well known example of an intermediate species is the substrate-enzyme complex $ES$ in the Michaelis-Menten mechanism,
 
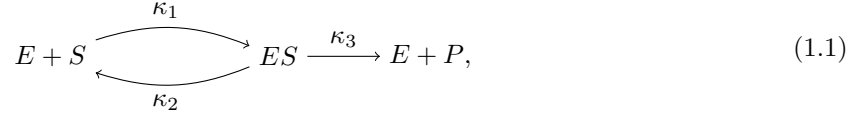
\begin{equation}\label{eq:michaelis_menten}
   \begin{tikzpicture}[baseline={(current bounding box.center)}]
    \node[state] (E+S) at (0,0) {$E+S$};
    \node[state] (ES) at (3,0) {$ES$};
    \node[state] (E+P) at (5,0) {$E+P$,};
    \path[->] 
     (E+S) edge[bend left] node[pos=.45]{$\kappa_1$} (ES)
     (ES) edge[bend left] node[pos=.55]{$\kappa_2$} (E+S)
           edge            node{$\kappa_3$} (E+P);
   \end{tikzpicture}
  \end{equation}
  but more complex intermediate structures might be considered as well.  The constants $\kappa_1$, $\kappa_2$ and $\kappa_3$ are rate constants and relate to the propensity of the reactions to occur. Various effects of the presence of intermediate species in reaction networks have recently  been studied; for example in relation to the number of steady states  \cite{gunawardena:unlimited,feliu:intermediates}, persistence \cite{freitas:persistence}, and approximation of stochastic trajectories \cite{cappelletti:intermediates}.
 
 To illustrate the scope of our results, consider the following system:
  
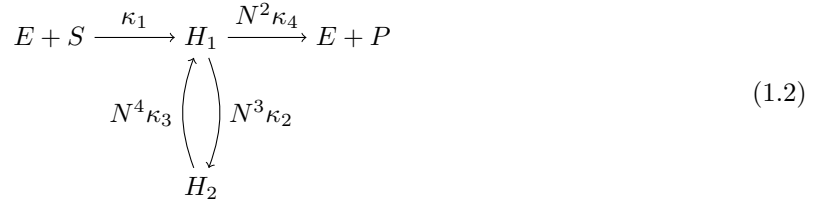
\begin{equation}\label{eq:example}
  \begin{tikzpicture}[baseline={(current bounding box.center)}]
   \node[state] (E+S) at (0,2) {$E+S$};
   \node[state] (\is_1) at (2,2) {$\is_1$};
   \node[state] (\is_2) at (2,0) {$\is_2$};
   \node[state] (E+P) at (4,2) {$E+P$};
   \path[->] 
    (E+S) edge node{$\kappa_1$} (\is_1)
    (\is_1) edge[bend left] node{$N^3\kappa_2$} (\is_2)
    (\is_2) edge[bend left] node{$ N^4\kappa_3$} (\is_1)
    (\is_1) edge node{$ N^2\kappa_4$} (E+P);
  \end{tikzpicture}
 \end{equation}
 Here, the species $\is_1$ and $\is_2$ are intermediate species, whose production and consumption delay the  enzymatic reaction that transforms a substrate $S$ into a product $P$.  If we denote by $S(t)$ the concentration of the substrate at time $t$ and so on, then the  dynamics of the system 
  is modelled by (assuming mass-action kinetics)
 \begin{align*}\label{eq:system_ode_example}
  \frac{d}{dt}S(t)&=-\kappa_1 E(t)S(t) \\
  \frac{d}{dt}E(t)&=-\kappa_1 E(t)S(t)+N^2\kappa_4\is_1(t) \\
  \frac{d}{dt}P(t)&=N^2\kappa_4\is_1(t)\\
  \frac{d}{dt}\is_1(t)&=\kappa_1 E(t)S(t)-N^3\kappa_2\is_1(t)+N^4\kappa_3\is_2(t)-N^2\kappa_4\is_1(t)\\
  \frac{d}{dt}\is_2(t)&=N^3\kappa_2\is_1(t)-N^4\kappa_3\is_2(t)
 \end{align*}
The number $N$ is supposed to be large, representing the separation in time-scales between reaction rates.  If $N$ is large, then the presence of the intermediate species $\is_1$ and $\is_2$ does not de facto  delay the production of the product $P$, since the intermediate species are consumed almost immediately after production. Our result allows us to  approximate the trajctories of the system \eqref{eq:example}, uniformly on  compact time intervals, by that of the following reduced reaction system
  \begin{equation}\label{eq:example_reduced}
  \begin{tikzpicture}[baseline={(current bounding box.center)}]
   \node[state] (E+S) at (0,2) {$E+S$};
   \node[state] (E+P) at (2,2) {$E+P$};
   \path[->] 
    (E+S) edge node{$\kappa_1$} (E+P);
  \end{tikzpicture}
 \end{equation}
 (also with mass-action kinetics). 
 
  In order to describe the dynamics of biologically realistic systems, we  consider a multiscale framework where not only the reaction rates scale, but also the species abundances scale in orders of $N$. For example, the reduction we performed for \eqref{eq:example} is still valid if the   substrate concentration  is of order $O(N)$, the enzyme concentration is of order $O(1)$ and the product concentration is of order $O(N^\gamma)$, where $\gamma\geq1$. As a consequence, the degradation rate of the substrate  and the enzyme is also of order $O(N)$. It is worth pointing out, that while the reaction rates that depend on $N$ in general will increase with $N$, our approach also allows for some reaction rates to become arbitrary small, as in Example \ref{ex:mu_to_infty} below.

All time and concentration scales are parameterised by a single parameter $N$, which corresponds to $1/\varepsilon$ in the context of Tikhonov's approach. In contrast to  previous work on reduction of deterministic reaction networks, our approach identifies a proper reduced reaction network as in \eqref{eq:example_reduced}, with rates potentially depending on $N$, and a limiting reaction network with rates independent on $N$; and not only an ODE system 
 approximating the species behaviour. In the case discussed above, the two reaction networks coincide, together with their kinetics, but  we will see examples where this is not the case. Furthermore, we prove uniform convergence of the trajectories of the original network to those of the reduced, as well as to those of the limiting reaction network, on compact time intervals.  In the particular case of \eqref{eq:example}, the reduction cannot be performed using Tikhonov's theorem.   The problem resides in the fact that the fast reactions are not of the same order of magnitude  
 and this cannot be accommodated in the setting of Tikhonov. Thus, it is not sufficient to categorise the reactions (or variables) as fast and slow, but the exact order of each reaction is important. For example, if the rate constant of $\is_2\to\is_1$ is changed from $N^4\kappa_3$ to $N\kappa_3$, then the intermediate structure causes a significant delay and \eqref{eq:example} cannot be reduced to \eqref{eq:example_reduced} (in this case, Assumption \ref{ass:rate_to_zero} below does not hold). This is caused by the cycle between the intermediate species $\is_1$ and $\is_2$.
 
 The paper is organised as follows. Section \ref{sec:def} contains background material on reaction networks and definitions. Section \ref{sec:int_effect} discusses intermediate species and some mathematical consequences of introducing reactions including intermediate species. In Section \ref{sec:multiscale}-\ref{sec:equivalent} we introduce the multiscale setting and the reduced reaction network, and in Section \ref{sec:limiting}, the limiting reaction network is introduced. The convergence results are stated in Section \ref{sec:theorems} with the proofs postponed to Section \ref{sec:proofs}. Section \ref{sec:discussion} relates our approach to other approaches and discusses problems arising when considering long term behaviours. Our approach is inspired by related work for stochastic reaction networks \cite{popovic:rescale, cappelletti:intermediates}.

 \section{Definitions and background}\label{sec:def}
  
 In this section we introduce  definitions and background material; for more details on reaction systems, see for example \cite{feinberg:review} and \cite{erdi:mathematical_models}.
 
  Let $\RR$, $\RR_{>0}$, and $\RR_{\geq0}$ be the set of real, positive real and non-negative real numbers, respectively. Also let $\NN$ be the set of natural numbers including 0. For any vector $v\in\RR^p$, we let $v_i$ be the $i$th component of $v$  and $\|v\|$ the Euclidean norm. We denote by $e$ the vector with all entries equal to one and by $e_i$, the $i$th unit vector. If $M$ is matrix (or vector), $M^\top$ denotes the transpose of $M$.
  Furthermore, for any set $A$, $|A|$ denotes the cardinality of $A$, and for any two sets $A$ and $B$, we let $A\setminus B$ be the set of elements that are  in $A$, but not in $B$.
 If  $u,v\in\RR^p$ are vectors and $N>0$ a scalar, then  $N^u$ denotes the vector with entries $N^{u_i}$ and  $N^uv$ denotes the vector with entries $N^{u_i}v_i$. 
 Finally, if $g, f\colon \NN\to\RR$ are functions, then $g(N)=O(f(N))$ denotes that $g$ is of order at most that of $f$, that is, $\limsup_{N\to \infty} |g(N)/f(N)|<\infty$, and $g(N)=\Theta(f(N))$ denotes that $g$ is of the same order as $f$, that is,  $0<\liminf_{N\to \infty} |g(N)/f(N)|$ and $\limsup_{N\to \infty} |g(N)/f(N)|<\infty$.
 
  A reaction network is a triple $(\Sp,\C,\R)$, where $\Sp$ is an ordered set $(S_k)_{1\leq k\leq |\Sp|}$, $\C$ is an ordered set $(y_i)_{1\leq i\leq |\C|}$ of linear combinations of elements of $\Sp$ on $\NN$, and $\R$ is a subset of $\C\times\C$, such that $(y_i,y_i)\notin\R$ for all  $y_i\in\C$. The elements of the set $\Sp$ are called species, the elements of the set $\C$ are called complexes, and  the elements of $\R$ are called reactions. The complexes are identified as vectors in $\RR^{|\Sp|}$. A reaction $(y_i,y_j)\in\R$ is denoted by $y_i\to y_j$.  
In  \eqref{eq:example}, there are $5$ species ($S,E,\is_1,\is_2,P$), $4$ complexes ($S+E, \is_1, \is_2, P+E$), and $4$ reactions. 

 The evolution of the species concentrations $x(t)\in\RR^n_{\geq0}$  for $t\ge 0$ is modelled as the solution to the ODE system
 \begin{equation}\label{eq:ODEplain}
 \frac{d}{dt}x(t)=\sum_{y_i\to y_j\in\R}(y_j-y_i)\lambda_{ij}(x(t)),
 \end{equation}
 with initial condition $x(0)\in\RR^n_{\ge 0}$,
 for some non-negative, non-zero functions $\lambda_{ij}\colon\RR^{|\Sp|}_{\geq0}\to\RR_{\ge 0}$, fulfilling that  $\lambda_{ij}(x)>0$ implies  $ x_k>0$, whenever $y_{ik}>0$, $1\le k\le |\Sp|.$  The latter assumption requires that a reaction only occurs in the presence of the involved species,  and therefore prevents the concentrations from becoming negative. 
  The functions $\lambda_{ij}(x)$ are called \emph{rate functions}, together they constitute a \emph{kinetics} $\K$ for $(\Sp,\C,\R)$, and the quadruple $(\Sp,\C,\R,\K)$ is called a \emph{(deterministic) reaction system}.  If 
 \begin{equation}\label{eq:mass-action}
 \lambda_{ij}(x)=\kappa_{ij}\prod_{k=1}^{|\Sp|}x_k^{y_{ik}}
 \end{equation}
  for all reactions, then the constants $\kappa_{ij}$ are referred to as \emph{rate constants} and the modelling regime  as \emph{(deterministic) mass-action kinetics}. In this case, the quadruple  $(\Sp,\C,\R,K)$ is called a \emph{(deterministic) mass-action system}.
  
For convenience, we define $\lambda_{ij}(x)=0$ whenever $y_i\to y_j$ is not a reaction in $\R$, in which case \eqref{eq:ODEplain} becomes
 \begin{equation*}
  \frac{d}{dt}x(t)=\sum_{1\leq i,j\leq |\C|}(y_j-y_i)\lambda_{ij}(x(t)),
 \end{equation*}
  
  Finally, we define intermediate species as in \cite{feliu:intermediates}. 

 \begin{definition}\label{def:intermediate_species}
  Let $(\Sp,\C,\R)$ be a reaction network and $\In=(\is_\ell)_{\ell\in\ii}$ be a subset of $\Sp$. We say that the species in $\In$ are \emph{intermediate species} (or simply \emph{intermediates}) if the following conditions hold:
  \begin{itemize}
   \item for each $\is_\ell\in\In$, the only complex involving $\is_\ell$ is $\is_\ell$ itself.
   (This implies that $\In\subseteq\C$.)
   \item for each $\is_\ell\in\In$, there is a directed path of complexes, such that 
   $$y_i\rightarrow \is_{\ell_1}\rightarrow\dots\rightarrow \is_\ell\rightarrow\dots\rightarrow \is_{\ell_n}\rightarrow y_j$$
  with $y_i, y_j\in\C\setminus\In$ and $\is_{\ell_m}\in\In$ for all $1\leq m\leq n$.
  \end{itemize}
 \end{definition}

 According to the Definition \ref{def:intermediate_species}, intermediate species always appear alone and with stoichiometric coefficient one. For example, the substrate-enzyme complex in the Michaelis-Menten mechanism \eqref{eq:michaelis_menten} and the species $\is_1$, $\is_2$ in \eqref{eq:example} meet Definition \ref{def:intermediate_species}. We denote by $\SR$, $\FP$, the subsets of $\C\setminus\In$ such that
 \begin{itemize}
  \item $y_i\in\SR$ if and only if $y_i\notin\In$ and there exists $\is_\ell\in\In$, such that $y_i\rightarrow \is_\ell\in\R$
  \item $y_j\in\FP$ if and only if $y_j\notin\In$ and there exists $\is_\ell\in\In$, such that $\is_\ell\rightarrow y_j\in\R$
 \end{itemize}
 Thus, $\SR$ consists of the complexes from which intermediates are produced and $\FP$ consists of the complexes to which intermediates are degraded.
 We refer to $\SR$ and to $\FP$, respectively, as the \emph{initial reactants} and the \emph{final products}. In general, the two sets can have non-empty intersection, as is the case for the Michaelis-Menten mechanism \eqref{eq:michaelis_menten}.

 For convenience, we index the sets $\Sp$  and $\C$, such that $S_\ell=y_\ell=\is_\ell$ for any intermediate $\is_\ell\in\In$. Further, we introduce the index sets $\sr$, $\ii$,  and $\fp$ of $\SR$, $\In$,  and $\FP$, respectively, such that
 $$\SR=\prg{y_i}_{i\in\sr},\quad\In=\prg{\is_\ell}_{\ell\in\ii}, \quad\FP=\prg{y_j}_{j\in\fp}.$$

  \section{Effects of the intermediate species}
 \label{sec:int_effect}
 
 The presence of intermediate species slows down any reaction path that proceeds through the formation of intermediates. Intuitively, the production and degradation of a sequence of intermediate species delay the synthesis of the final product, while in a model without intermediates this synthesis would happen without any delay.

 Let $(\Sp, \C, \R,\K)$ be a reaction system with a set of intermediate species $\In\subseteq\Sp$. To investigate the effects of the presence of intermediate species in detail, we make the following assumption:

  \begin{assumption}[Reaction rates and intermediates]\label{ass:rate_deterministic}
   The consumption of the intermediates is governed by mass-action kinetics, namely, for any $\ell,\ell^\prime\in\ii$ and $j\in\fp$,
   $$\lambda_{\ell j}(x)=\kappa_{\ell j}x_\ell\quad\text{and}\quad\lambda_{\ell\ell^\prime}(x)=\kappa_{\ell \ell^\prime}x_\ell,$$
   for some  constants $\kappa_{\ell j}$, $\kappa_{\ell \ell^\prime}>0$. Furthermore, we assume that all other reaction rates do not depend on $\is_\ell$ in $\In$.
  \end{assumption}
  
 Let $\pi$ be the projection onto the non-intermediate species, and $\rho$ the projection onto the intermediate species. To ease the notation, let $\xni=\pi(x)$ and $\xin=\rho(x)$, such that $x=(\xni,\xin)$ for $x\in\RR^{|\Sp|}$.
  Under Assumption \ref{ass:rate_deterministic}, the rates of reactions that are not consuming intermediates depend on  $x$ only through $\xni$. Hence, with a slight abuse of notation, we write
  $$\lambda_{ij}(x)=\lambda_{ij}(\xni),\quad i\notin\ii.$$
  
  For any $i\in \sr$, consider the labelled directed graph $\mathcal{G}_i^{\xni}$ with node set $\In\cup\{\star\}$ and labelled edge set given by:
  \smallskip
  \begin{equation}\label{eq:G_i^x}
   \text{\storestyleof{itemize}
  \begin{listliketab}
   \begin{tabular}{Lll}
    \textbullet & $\is_\ell\xrightarrow[\phantom{\sum_{j\in\fp} \kappa_{\ell j}}]{\kappa_{\ell\ell^\prime}}\is_{\ell^\prime}$ & if $\kappa_{\ell\ell^\prime}\neq 0$ and $\ell\neq\ell^\prime$ \\
    \textbullet & $\is_\ell\xrightarrow{\sum_{j\in\fp} \kappa_{\ell j}}\star$ & if $\displaystyle\sum_{j\in\fp} \kappa_{\ell j}\neq 0$ \\
    \textbullet & $\hspace*{0.25cm}\star\xrightarrow[\phantom{\sum_{j\in\fp} \kappa_{\ell j}}]{\lambda_{i \ell}(\xni)}\is_\ell$ & if $\lambda_{i \ell}(\xni)\neq0$
   \end{tabular}
  \end{listliketab}}
  \end{equation}
By Definition \ref{def:intermediate_species},  there is directed path from any $\is_\ell$ to $\star$. Even though all intermediate species  are produced,  there might not be a directed path from $\star$ to an intermediate species $\is_\ell$, as $\lambda_{i\ell}(\xni)$ could be zero for some reaction $y_i\to\is_\ell$  and the particular choice of $\xni$. Hence  $\mathcal{G}_i^{\xni}$ might not be strongly connected. 
 
  If we order the nodes of the graph such that $\star$ is the last one, (the transpose of) the Laplacian matrix of the graph \eqref{eq:G_i^x} takes the form
  {\renewcommand\arraystretch{2}
   \begin{equation}\label{eq:laplacian}
    L_i^{\xni}=\left[\begin{array}{C|D}
                   -\La    & -\lambda_i(\xni) \\ \hline
                   e^\top \La  & \sum_{\ell\in\ii}\lambda_{i\ell}(\xni) \\
                  \end{array} \right],
   \end{equation}}where, for any $\ell,\ell^\prime\in\ii$,
  $$\La_{\ell\ell^\prime}=\begin{cases}
                         \kappa_{\ell^\prime\ell}&\text{if }\ell\neq\ell^\prime\\
                         -\sum_{h\in\ii\cup\fp}\kappa_{\ell h}&\text{if }\ell=\ell^\prime,
                        \end{cases}$$
and
  \begin{equation}\label{eq:lambda_i}
  \lambda_{i\cdot}(\xni)=(\lambda_{i\ell}(\xni))_{\ell\in\ii}.
  \end{equation}
Finally, we define the vector $\Lambda(\xni)$ of length $|\In|$, by 
  $$\Lambda_\ell(\xni)=\sum_{i\in\sr}\lambda_{i\ell}(\xni).$$
  With these definitions,  $\xin(t)$ is a solution to
  $$\frac{d}{dt}\xin(t)=\La \xin(t) + \Lambda(\xni(t)),$$
  which implies that 
    \begin{equation}\label{eq:solution_for_\is}
   \xin(t)=\exp\pr{\La t}\xin(0)+\int_0^t\exp\pr{\La (t-s)}\Lambda(\xni(s)) ds.
  \end{equation}
 The non-intermediate species evolve according to
 \begin{align*}\label{eq:diff_eq_for_X}
  \frac{d}{dt}\xni(t)&=\sum_{\substack{\ell\in\ii \\ j\in\fp}}y_j\kappa_{\ell j}x_\ell(t) + \sum_{\substack{i\notin\ii \\ 1\leq j\leq |\C|}}\pi(y_j-y_i)\lambda_{ij}(\xni(t)) \notag\\
  &=\sum_{j\in\fp}y_j\vkappaj\xin(t) + \sum_{\substack{i\notin\ii \\ 1\leq j\leq |\C|}}\pi(y_j-y_i)\lambda_{ij}(\xni(t)),
 \end{align*}
  where $\vkappaj=(\kappa_{\ell j})_{\ell\in V}$ are row vectors.  
  Under the assumption $\xin(0)=0$, it follows from \eqref{eq:solution_for_\is} that
 \begin{equation}\label{eq:delayed_ODE}
  \frac{d}{dt}\xni(t)=\sum_{j\in\fp}y_j\vkappaj\int_0^t\exp\pr{\La (t-s)}\Lambda(\xni(s)) ds + \sum_{\substack{i\notin\ii \\ 1\leq j\leq |\C|}}\pi(y_j-y_i)\lambda_{ij}(\xni(t)).
 \end{equation}
 The above is a system of delayed differential equations with a distributed delay, in the sense of \cite{kuang:delay}. In particular, \eqref{eq:delayed_ODE} does not depend explicitly on the abundance of the intermediate species. 
 
  \begin{remark}\label{rem:L_not_singular}
 The matrix $L$ is  invertible, which  follows from standard graph theory and Gershgorin theorems.   
  By potentially  changing the order of the intermediate species, $L$ can be transformed into a block diagonal matrix with irreducible diagonal blocks. Moreover, in each diagonal block there exists at least one column for which the diagonal entry is strictly smaller than the sum of the other entries. Such a column corresponds to an intermediate species that degrades to a final product.  Therefore, we can conclude by the first and the third Gershgorin theorem that all eigenvalues of $L$ have negative real part \cite{book:num_an}. In particular, zero cannot be an eigenvalue of $L$, which is therefore invertible.  
 \end{remark}

  \section{The multiscale setting}\label{sec:multiscale}
 
  Consider a reaction system $(\Sp,\C,\R,\K)$ with a set of intermediate species $\In\subseteq\Sp$. Our aim is to study the asymptotic behaviour of the trajectories of the system under the assumption that the consumption rates of the intermediate species are high (technically, what we will require is slightly different and expressed in Assumption \ref{ass:rate_to_zero}). Formally, we introduce a sequence of kinetics $\K^N$, indexed by $N\in\NN$, and let $x^N\!(t)$ denote the solution of the  reaction system $(\Sp,\C,\R,\K^N)$ with a given initial condition $x^N\!(0)\in\RR^{|\Sp|}_{\geq0}$. We assume the kinetics $\K^N$, $N\in\NN$, satisfy Assumption \ref{ass:rate_deterministic}.

  In the typical biological context the abundance of the non-intermediate species might  differ by orders of magnitude. In  addition the rate of their degradation might likewise differ. To accommodate this into the setting we introduce the sets
 $$\R_0=\{y_i\to y_j\in\R\,\vert\,y_i,y_j\not\in V\},\quad \R_1=\{y_i\to y_j\in\R\,\vert\,y_i\not\in V\}$$
  of all reactions not involving intermediate species and all reactions not consuming intermediate species, respectively. Clearly, $\R_0\subseteq \R_1$.
  
 We define two real vectors $\alpha\in\RR^{|\Sp\setminus\In|}$ and $\beta\in\RR^{|\R_1|}$,  which record the orders of magnitude of the non-intermediate species and of the reaction rates, respectively.   We assume that the kinetics $\K^N$, $N\in\NN$, are such that for any $y_i\to y_j\in\R_1$,
  \begin{equation}\label{eq:multiscale_beta}
   \lim_{N\to\infty}N^{-\beta_{ij}}\lambda^N_{ij}(N^\alpha\xni)=\lambda_{ij}(\xni),
  \end{equation}
   uniformly in  $\xni$ on compact sets of $\RR^{|\Sp\setminus\In|}$, for some locally Lipschitz function $\lambda_{ij}\colon \RR^{|\Sp\setminus\In|}\to \RR_{\ge 0}$, which is non-zero  for $y_i\to y_j\in\R_1\setminus\R_0$ (that is, there exists $\xni$, such that the function is non-zero). The latter is a natural requirement and emphasises that the scaling ideally should be such that the reaction rate  persists in the limit. Technically we only need this requirement for reactions in $\R_1\setminus\R_0$, see Remark \ref{rem:equivalent}.

 \begin{remark}\label{rem:mass-action}
 For mass-action kinetics, there is a natural choice of $\beta\in\RR^{|\R_1|}$. Assume 
  the rate constants in \eqref{eq:mass-action} takes the form $\kappa^N_{ij}=N^{\eta_{ij}}\kappa_{ij}$, where $\kappa_{ij}>0$. Then, from \eqref{eq:multiscale_beta},
 $$N^{-\beta_{ij}} \lambda^N_{ij}(N^\alpha\xni)=N^{-\beta_{ij}} N^{\eta_{ij}}\kappa_{ij}\prod_{k=1}^{|\Sp|}(N^{\alpha_k}\xni_k)^{y_{ik}}=N^{-\beta_{ij}+\eta_{ij}+\langle \alpha,y_i\rangle}\kappa_{ij}\prod_{k=1}^{|\Sp|}\xni_k^{y_{ik}},$$
 where $\langle ,\rangle$ denotes the scalar product. Hence, the expression converges for large $N$ to a non-zero limit if and only if $\beta_{ij}=\eta_{ij}+\langle \alpha,y_i\rangle$.
 \end{remark}
 
 \section{The reduced reaction system}\label{sec:reduced}
 
 For simplicity we define the following.
 \begin{definition}
   We say that a complex $y$ reacts to another complex $y'$ through intermediates, and write $y\rti y'$, if one of the following possibilities occurs.
   \begin{itemize}
    \item $y\in\In$ and either $y=y'$ or $y\to y'\in\R$;
    \item there exists a sequence of intermediate species $\is_{\ell_1}$, $\is_{\ell_2},\dots, \is_{\ell_n}$ such that
 $$y\to \is_{\ell_1}\to\dots\to\is_{\ell_n}\to y'.$$
   \end{itemize}
 \end{definition}
 
 Consider a reaction network $(\Sp,\C,\R)$ with a set of intermediate species $\In\subseteq\Sp$. We define the reduced reaction network  as in \cite{feliu:intermediates} and \cite{cappelletti:intermediates}, that is,  as
 \begin{equation}\label{eq:reduced_reaction_network}
  (\Sp\setminus\In,\C\setminus\In,\Rr),
 \end{equation}
 where $\Rr$ consists of  the reactions $y_i\rightarrow y_j$, such that either $y_i\rightarrow y_j$ is an element of $\R$ not involving any intermediate, or the complex $y_j$ reacts to $y_i$ through intermediates. Formally, 
$$\Rr=\R_0\,\cup\,\R_1^r,$$
where
$$\R_1^r=\{y_i\rightarrow y_j\mid  y_i\rti y_j, \, y_i\in\SR,y_j\in\FP\}.$$
Note that $\R_0$ and $\R_1^r$ may have non-empty intersection.
 For  the reaction networks \eqref{eq:michaelis_menten} and \eqref{eq:example}, we have $\Rr=\{E+S\to E+P\}$.
   
If the original reaction network $\R$ is equipped with the kinetics $\K^N$, we introduce a kinetics $\widetilde{\K}^N$ for the reduced reaction network\eqref{eq:reduced_reaction_network}, induced by $\R$ and $\K^N$ \cite{feliu:intermediates}.
To define $\widetilde{\K}^N$, we  need some further terminology. 
 
 Let the labelled directed graph $\mathcal{G}_i^{\xni,N}$ be as in \eqref{eq:G_i^x}, where we make the dependence on $N$ explicit. Let $\mathcal{T}_i^{\xni,N}(\cdot)$ be the set of labelled spanning trees of $\mathcal{G}_i^{\xni,N}$ rooted at the argument, and let $\sigma^N\!(\cdot)$ be the product of the edge labels of the tree in the argument. To be precise, we say that a tree is rooted at a node if all the directed edges are directed towards the root. Define 
 \begin{equation}\label{eq:definition_of_mu}
  \mu^N_{i\ell}(\xni)=\frac{\sum_{\tau\in\mathcal{T}_i^{\xni,N}\!(\is_\ell)}\sigma^N\!(\tau)}{\sum_{\tau\in\mathcal{T}_i^{\xni,N}\!\pr{\star}}\sigma^N\!(\tau)}.
 \end{equation}
The denominator is always strictly positive, since there is at least one spanning tree rooted at $\star$ and all labels are positive.  Furthermore, $\sigma^N\!(\tau)$ of such a spanning tree is independent of  $\xni\in\RR^{|\Sp\setminus\In|}$, see \eqref{eq:G_i^x}. In contrast, the numerator might be zero if there is not a spanning tree rooted at $H_\ell$. This might be the case if $y_i\to\is_\ell\notin\R$ or $\lambda^N_{i \ell}(\xni)=0$.
 
 The kinetics $\widetilde{\K}^N$ of the reduced reaction system is defined by the rate functions
 \begin{equation}\label{eq:reduced_rates_deterministic}
  \lambdarN_{ij}(\xni)=\lambda^N_{ij}(\xni)+\sum_{\ell\in\ii} \kappa^N_{\ell j}\mu^N_{i\ell}(\xni)\quad\text{for}\quad y_i\to y_j\in\Rr.
 \end{equation}  
 Similarly to the original system, the reduced reaction system is a multiscale system, and 
we assume there is a  vector $\betar\in\RR^{|\Rr|}$, such that for any compact set $\Gamma\in\RR^{|\Sp\setminus\In|}$ and any $y_i\to y_j\in\Rr$
\begin{equation}\label{eq:multiscale_beta_reduced}
 \lim_{N\to\infty}\sup_{\xni\in\Gamma}N^{-\betar_{ij}}\lambdarN_{ij}(N^\alpha\xni)<\infty
\end{equation}
Due to \eqref{eq:multiscale_beta} and \eqref{eq:reduced_rates_deterministic}, there is always $\betar_{ij}$, such that \eqref{eq:multiscale_beta_reduced} holds. This also follows straightforwardly from \eqref{eq:sum} below.  In particular, If $y_i\to y_j\in\R_0\setminus \R^r_1$, then it also follows that $\betar_{ij}$ can be chosen such that $\betar_{ij}\ge \beta_{ij}$.

 We restrict the analysis to the case in which  the species abundances balance the rate by which they change in the sense of the following assumption.  Essentially, it restricts how large we can choose $\beta^r_{ij}$ in \eqref{eq:multiscale_beta_reduced}.

  \begin{assumption}\label{ass:single_scale}
  For any reaction $y_i\to y_j$ in $\R^r$, we have
   $$\lim_{N\to\infty} N^{\betar_{ij}}\|N^{-\alpha}(y_j-y_i)\|<\infty,$$
   and $\betar_{ij}\ge \beta_{ij}$ for $y_i\to y_j\in \R_0$.
  \end{assumption}

In particular, the assumption implies that $N^{\beta_{ij}}\|N^{-\alpha}(y_j-y_i)\|$ has finite limit for all $y_i\to y_j\in \R_0$. We introduce  two examples that will serve as running examples.
\begin{example}[part 1]\label{ex:michaelis_menten}
 Consider the Michaelis-Menten mechanism, taken with mass-action kinetics:
 
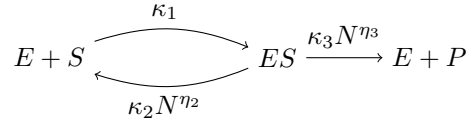
\begin{equation*}
   \begin{tikzpicture}[baseline={(current bounding box.center)}]
    \node[state] (E+S) at (0,0) {$E+S$};
    \node[state] (ES) at (3,0) {$ES$};
    \node[state] (E+P) at (5,0) {$E+P$};
    \path[->] 
     (E+S) edge[bend left] node[pos=.45]{$\kappa_1$} (ES)
     (ES) edge[bend left] node[pos=.55]{$\kappa_2N^{\eta_2}$} (E+S)
           edge            node{$\kappa_3N^{\eta_3}$} (E+P);
   \end{tikzpicture}
  \end{equation*}
We take $\In=\{ES\}$ to be the set of intermediates. Then $\SR=\{E+S\}$ and $\FP=\{E+S, E+P\}$.  The graph \eqref{eq:G_i^x} is 
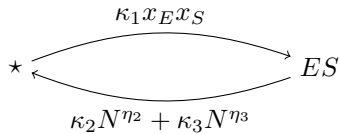
 \begin{equation*}
   \begin{tikzpicture}[baseline={(current bounding box.center)}]
    \node[state] (star) at (0,0) {$\star$};
    \node[state] (ES) at (4,0) {$ES$};
    \path[->] 
     (star) edge[bend left] node{$\kappa_1x_Ex_S$} (ES)
     (ES) edge[bend left] node{$\kappa_2N^{\eta_2}+\kappa_3N^{\eta_3}$} (star);
   \end{tikzpicture}
  \end{equation*}
  Therefore,
  $$\mu^N_{E+S, ES}(x)=\frac{\kappa_1x_Ex_S}{\kappa_2N^{\eta_2}+\kappa_3N^{\eta_3}},$$
and it follows that the reduced reaction system is given by the following mass-action system (see \eqref{eq:reduced_rates_deterministic}),
 \begin{equation*}
   \begin{tikzpicture}[baseline={(current bounding box.center)}]
    \node[state] (E+S) at (0,0) {$E+S$};
    \node[state] (E+P) at (3,0) {$E+P$};
    \path[->] 
     (E+S) edge node{$\frac{\kappa_1\kappa_3N^{\eta_3}}{\kappa_2N^{\eta_2}+\kappa_3N^{\eta_3}}$} (E+P);
   \end{tikzpicture}
  \end{equation*}
  Finally, Assumption \ref{ass:single_scale} is satisfied with $\alpha_E=0$, $\alpha_S<\max\{\eta_2,\eta_3\}$ and $\alpha_P=\min\{\alpha_S,\alpha_S+\eta_3-\eta_2\}$. The order of magnitude of the reaction rates are $\beta_{E+S\to ES}=\alpha_S$ and $\betar_{E+S\to E+P}=\alpha_S+\eta_3-\max\{\eta_2,\eta_3\}$, see Remark \ref{rem:mass-action}.
\end{example}

\begin{example}[part 1]
 Consider the system \eqref{eq:example} with set of intermediates $\In=\{\is_1,\is_2\}$. As in the previous example, there is only one initial reactant, $\SR=\{E+S\}$, and one final product, $\FP=\{E+P\}$. The graph \eqref{eq:G_i^x} is given by
 
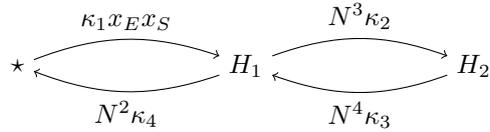
\begin{equation*}
   \begin{tikzpicture}[baseline={(current bounding box.center)}]
    \node[state] (star) at (0,0) {$\star$};
    \node[state] (is_1) at (3,0) {$\is_1$};
    \node[state] (is_2) at (6,0) {$\is_2$};
    \path[->] 
     (star) edge[bend left] node{$\kappa_1x_Ex_S$} (is_1)
     (is_1) edge[bend left] node{$N^2\kappa_4$} (star)
     (is_1) edge[bend left] node{$N^3\kappa_2$} (is_2)
     (is_2) edge[bend left] node{$N^4\kappa_3$} (is_1);
   \end{tikzpicture}
  \end{equation*}
 In this case,
 $$\mu^N_{E+S,\is_1}(x)=\frac{N^4\kappa_1\kappa_3x_Ex_S}{N^6\kappa_3\kappa_4}.$$
 Since the reaction constant of the reaction $H_1\to E+P$ is $N^2\kappa_4$, the reduced reaction system is given by the mass-action system (see \eqref{eq:reduced_rates_deterministic})
 \begin{equation*}
   \begin{tikzpicture}[baseline={(current bounding box.center)}]
    \node[state] (E+S) at (0,0) {$E+S$};
    \node[state] (E+P) at (3,0) {$E+P$};
    \path[->] 
     (E+S) edge node{$\kappa_1$} (E+P);
   \end{tikzpicture}
  \end{equation*}
 Note that the kinetics of the reduced reaction system does not depend on $N$. Finally, Assumption \ref{ass:single_scale} 
 is satisfied if, for example, $\alpha_E=0$, $\alpha_S\leq\alpha_P$. In this case the orders  of the reaction rates are given by $\beta_{E+S\to \is_1}=\betar_{E+S\to E+P}=\alpha_S$, see Remark \ref{rem:mass-action}.
 \end{example}

\section{Equivalent description of the reduced reaction system}\label{sec:equivalent}

The definition of the reduced reaction system is that of   
\cite{feliu:intermediates}.  It is also used in \cite{cappelletti:intermediates}, where   a  probabilistic interpretation of the kinetics $\K^N$ is given in the context of stochastic reaction systems. Specifically, it is shown that
\begin{equation*}
\kappa^N_{\ell j}\mu^N_{i\ell}(x)=\lambda^N_{i\ell}(x)\pi^N_{\ell j},
\end{equation*}
where $\pi^N_{\ell j}$ is the probability that $y_j$ is the final product eventually created by a single molecule of the intermediate species $\is_\ell$. In our context, $\pi^N_{\ell j}$ might be interpreted as the fraction of the concentration of $\is_\ell$,   that  is converted into $y_j$.

 In particular the relation implies that
\begin{equation}\label{eq:sum}
 \sum_{j\in\fp}\kappa^N_{\ell j}\mu^N_{i\ell}(x)= \lambda^N_{i\ell}(x). 
\end{equation}

The next lemma  concerns the inverse of the matrix $\La^N$, which exists by Remark \ref{rem:L_not_singular}. A similar result appears in \cite{feliu:intermediates}, and we give the proof in Section \ref{sec:proofs} for completeness.

 \begin{lemma}\label{lem:mu_and_inverse}
  Let $\mu^N_{i\ell}(x)$ be defined as in \eqref{eq:definition_of_mu}. We have that
  $$-\pr{\pr{\La^N}^{\!-1}\lambda_{i\cdot}^N\!(\xni)}_{\ell}=\mu^N_{i\ell}\!\pr{\xni},$$
  where $\ell$ on the left side indicates the $\ell$th entry  and $\lambda_{i\cdot}^N(\xni)$ is as in \eqref{eq:lambda_i}.
\end{lemma}
 \begin{remark}\label{rem:equivalent} From Lemma \ref{lem:mu_and_inverse} and equations   \eqref{eq:multiscale_beta}, \eqref{eq:reduced_rates_deterministic}, and \eqref{eq:multiscale_beta_reduced}, it follows that
  $$\limsup_{N\to\infty}\,-N^{-\betar_{ij}}\sum_{\ell, \ell'\in\ii} \kappa^N_{\ell' j}\pr{\pr{\La^N}^{\!-1}e_{\ell}N^{\beta_{i\ell}}}_{\ell'}<\infty,$$
 where  $\lambda^N_{i\ell}(N^\alpha \xni)=\Theta(N^{\beta_{i\ell}})$ for some $\xni\in \RR_{\geq0}^{|\Sp\setminus\In|}$, by \eqref{eq:multiscale_beta}.
 Therefore Assumption \ref{ass:single_scale}  implies that
 $$\limsup_{N\to\infty}\,-N^{\beta_{i\ell}-\alpha}(y_j-y_i)\vkappaj^N\pr{\La^N}^{-1}\!e_\ell <\infty$$
  for any $i\in\sr$, $j\in\fp$ and $\ell\in\ii$, such that $y_i\to\is_\ell\in\R$.
 \end{remark}
 
\section{The limiting reaction system}\label{sec:limiting}

The assumptions made in the previous sections allow us to approximate the dynamics of the original system to that of the reduced reaction system for any $N$ (Proposition \ref{prop:convergence}). However, it is also of interest to study the limit as $N$ tends to infinity and to obtain a limiting reaction system that is independent of $N$. For this, we need a stronger assumption than \eqref{eq:multiscale_beta_reduced}.

\begin{assumption}\label{ass:limiting_system}
 For any reaction $y_i\to y_j\in\Rr$, there exists a locally Lipschitz function $\lambdal_{ij}\colon\RR_{\geq0}^{|\Sp\setminus\In|}\to\RR_{\geq0}$, such that
 $$\lim_{N\to\infty}N^{-\betar_{ij}}\lambdarN_{i j}(N^\alpha\xni)=\lambdal_{ij}(\xni),$$
 uniformly on compact sets of $\RR_{\geq0}^{|\Sp\setminus\In|}$.
\end{assumption}

The assumption trivially implies \eqref{eq:multiscale_beta_reduced}. Under Assumption \ref{ass:limiting_system}, we can define a limiting reaction network. We begin by introducing a new set of complexes with cardinality  possibly different from that of $\C$. For any reaction $y_i\to y_j\in\Rr$, we define the complexes
$$y_i^{(i,j)}=\lim_{N\to\infty}N^{\betar_{ij}}N^{-\alpha}y_i\quad\text{and}\quad y_j^{(i,j)}=\lim_{N\to\infty}N^{\betar_{ij}}N^{-\alpha}y_j\;,$$
and let
$$\Cl=\{y_i^{(i,j)},y_j^{(i,j)}\colon y_i\to y_j\in\Rr\}, \quad \Rl=\{y_i^{(i,j)}\to y_j^{(i,j)}\colon y_i\to y_j\in\Rr\}.$$

The limiting reaction network is  defined as $(\Sp\setminus\In,\Cl,\Rl)$. A kinetics $\Kl$ for the limiting reaction network is defined by  the functions $\lambdal_{ij}(\xni)$, introduced in Assumption \ref{ass:limiting_system}, such that $\lambdal_{ij}(\xni)$ is the rate function of the reaction $y_i^{(i,j)}\to y_j^{(i,j)}$. Note that the rate function $\lambdal_{ij}(\xni)$ may be the constantly zero (as in Example \ref{ex:mu_to_infty} below), in which case we remove the reaction $y_i^{(i,j)}\to y_j^{(i,j)}$ from the reaction network. The  reaction system $(\Sp\setminus\In,\Cl,\Rl, \Kl)$ with such reactions removed, is called the limiting reaction system. 

\setcounter{example}{0}
\begin{example}[part 2]
 Since the enzyme $E$ does not change in any reaction of the reduced reaction system, its concentration remains constant over time.  As the reduced reaction network has only one reaction, we let $\betar_{E+S\to E+P}=\betar$.  To compute the limiting reaction system, recall that $\betar=\alpha_S+\eta_3-\max\{\eta_2,\eta_3\}$.  Assuming mass-action kinetics, we have the following different cases with $z_E(0)=E_0$:
 \begin{center}
 \begin{tabular}{c|c}
  \toprule
  Condition & Limiting system\\
  \toprule
  $\eta_2>\eta_3$, $\alpha_P>\betar$ & $\emptyset$\\
  & in this case $z(t)=z(0)$\\
  \midrule
  $\eta_2>\eta_3$, $\alpha_P=\betar$ & $0\xrightarrow{\frac{\kappa_1\kappa_3z_S(0)E_0}{\kappa_2}}P$\\
  \midrule
  $\eta_2=\eta_3$, $\alpha_P>\betar$ & $S\xrightarrow{\frac{\kappa_1\kappa_3E_0}{\kappa_2+\kappa_3}}0$\\
  \midrule
  $\eta_2=\eta_3$, $\alpha_P=\betar$ & $S\xrightarrow{\frac{\kappa_1\kappa_3E_0}{\kappa_2+\kappa_3}}P$\\
  \midrule
  $\eta_2<\eta_3$, $\alpha_P>\betar$ & $S\xrightarrow{\kappa_1E_0}0$\\
  \midrule
  $\eta_2<\eta_3$, $\alpha_P=\betar$ & $S\xrightarrow{\kappa_1E_0}P$\\
  \bottomrule
 \end{tabular}
 \end{center}
\end{example}

\begin{example}[part 2] The limiting reaction system coincides with the reduced reaction system.
\end{example}

\section{Convergence results}\label{sec:theorems}
 
 Our aim is to approximate the  evolution $\xni^N(t)$ of the non-intermediate species by a solution $z^N(t)$ of the reduced reaction system, or by a solution $z(t)$ of the limiting reaction system,  when the latter is defined. Specifically, we are interested in uniform convergence of the solutions on compact time intervals, in the sense of Proposition \ref{prop:convergence} and Theorem \ref{thm:convergence} below. In order to obtain such convergence, we need to make a key hypothesis that relates to the speed of consumption of the intermediate species. For convenience, we define the following quantity, for any $j\in\fp$ and $\ell\in\ii$:
 $$a_j=\min_{k\colon y_{jk}\neq0}\alpha_k,\qquad  \beta^*_\ell=\max_{i\in\sr\colon\! y_i\to\is_\ell\in\R}\,\,\beta_{i\ell},$$
 with the convention that the minimum (maximum) over the empty set is (minus) infinity. 
 \begin{assumption}\label{ass:rate_to_zero}
  For any $\ell\in\ii$, such that $y_i\to\is_\ell$ for some $i\in\sr$, and for any $\ell'\in\ii,\quad j\in\fp$, such that $\is_\ell\rti \is_{\ell'}$ and $\is_{\ell'}\rti y_j$, we assume 
  $$\lim_{N\to\infty}
  N^{\beta^*_\ell-a_j}\, e_{\ell'}^\top \exp\pr{N^{a_j-\beta^*_\ell}\varepsilon \La^N} e_{\ell}  =0.$$
 \end{assumption}
 
 \begin{remark}\label{rem:prob}
  Assumption \ref{ass:rate_to_zero} is very similar to what is required in \cite{cappelletti:intermediates} in order to have convergence of the evolution of a stochastic system with intermediates to that of a reduced reaction system without intermediates.   In particular, \cite{cappelletti:intermediates} offers a probabilistic interpretation of Assumption \ref{ass:rate_to_zero}. Let $\tau_\ell^N$ denote the random time until a molecule of $\is_\ell$ is  transformed into a non-intermediate complex, assuming a stochastic kinetics similar to  $\K^N$. The time $\tau_\ell^N$ then follows a phase-type distribution conditioned on a initial distribution, and Assumption \ref{ass:rate_to_zero} is then implied by
  \begin{equation}\label{eq:equivalence_stochastic}
   \lim_{N\to\infty}
  N^{\beta^*_\ell-a_j} P\pr{\tau_\ell^N>N^{a_j-\beta^*_\ell}\varepsilon}
  =0
  \end{equation}
  for any scalar $\varepsilon>0$, any $\ell\in\ii$, such that $y_i\to\is_\ell$ for some $i\in\sr$, and any $j\in\fp$, such that $\is_\ell\rti y_j$. Here $P$ denotes the probability of the event. Such an implication might be useful, as in some cases \eqref{eq:equivalence_stochastic} is easier to check than Assumption \ref{ass:rate_to_zero}.
 \end{remark}
 
  We give here some particular cases under which Assumption \ref{ass:rate_to_zero} holds. These  cases arise frequently in biological applications, so an explicit treatment  may be useful. Furthermore, if the concentrations of the non-intermediate species are of the same order of magnitude, as well as the propensities of the reactions transforming them, then we can simply consider $\alpha_k=0$ and $\beta_{ij}=0$ (for the relevant indices). In this particular case, it follows that $a_j=0$ for any $j\in\fp$ and $\beta^*_\ell=0$ for any $\ell\in\ii$ with $y_i\to\is_{\ell}$ for some $i\in\sr$.
 
 We first consider the particular case in which all the intermediate species have the same order of degradation.
 \begin{proposition}\label{prop:same_N}
  Assume that there exists $\gamma\in\RR$ such that for any $N\geq1$, and any $\ell\in\ii$ and $h\in\ii\cup\fp$, we have $\kappa^N_{\ell h}=N^\gamma\kappa_{\ell h}$ for some non-negative constants $\kappa_{\ell h}$. Moreover, assume that $\gamma>\beta^*_\ell-a_j$ for any $\ell\in\ii$, such that $y_i\to\is_\ell$ for some $i\in\sr$, and any   $j\in\fp$, such that $\is_\ell\rti y_j$. Then Assumption \ref{ass:rate_to_zero} holds.
 \end{proposition}
 \begin{proof}
 If $\kappa^N_{\ell h}=N^\gamma\kappa_{\ell h} $ for $\ell\in\ii$ and $h\in\ii\cup\fp$, then necessarily
 $$\La^N=N^{\gamma}\La,$$
 where for simplicity we let $\La=\La^1$. Following the argument in Remark \ref{rem:L_not_singular}, we conclude that the eigenvalues of $\La$ have negative real part. Then, it follows that there exist two positive constants $\Gamma_0$ and $\Gamma_1$ such that for any $t>0$
 $$\|\exp(tL)\|\leq \Gamma_0 e^{-\Gamma_1 t}.$$
 Hence, for any $\ell, \ell'\in\ii$
 \begin{align*}
  |N^{\beta^*_\ell-a_j}e_{\ell'}^\top \exp\pr{N^{a_j-\beta^*_\ell}\varepsilon \La^N} e_{\ell}|&=N^{\beta^*_\ell-a_j}|e_{\ell'}^\top \exp\pr{N^{a_j-\beta^*_\ell+\gamma}\varepsilon \La} e_{\ell}|\\
  &\leq N^{\beta^*_\ell-a_j} \Gamma_2 e^{-\Gamma_1 N^{a_j-\beta^*_\ell+\gamma}\varepsilon},
 \end{align*}
 for some positive constant $\Gamma_2$. This concludes the proof because  $\gamma>\beta^*_\ell-a_j$ by assumption.
 \end{proof}
 
 The second  case we consider deals with the absence of cycles in the intermediate structures, in the sense specified in the following proposition.
 \begin{proposition}\label{prop:tree}
  Assume there does not exist a directed path of the form $\is_{\ell_1}\to\dots\to\is_{\ell_n}\to \is_{\ell_1}$, for any sequence of intermediate species $\is_{\ell_1}$, $\is_{\ell_2},\dots, \is_{\ell_n}$. Moreover, assume that
  \begin{equation}\label{eq:tree}
   \lim_{N\to\infty}N^{a_j-\beta^*_{\ell}}\sum_{h\in\ii\cup\fp}\kappa^N_{\ell' h}=0
  \end{equation}
  for any $\ell\in\ii$ such that $y_i\to\is_\ell$ for some $i\in\sr$, and for any $\ell'\in\ii$, $j\in\fp$ such that $\is_{\ell}\rti \is_{\ell'}$ and $\is_{\ell'}\rti y_j$. Then Assumption  \ref{ass:rate_to_zero} holds.
 \end{proposition}
 \begin{proof}
  Note that by \eqref{eq:tree} the growth of the constants $\kappa^N_{\ell h}$ is at most polynomial in $N$. 
 
  Potentially by reordering the intermediate species, the matrix $\La^N$ is lower triangular. Indeed, by assumption it is not possible that $\is_{\ell}\rti \is_{\ell'}$ and $\is_{\ell'}\rti \is_{\ell}$ for two intermediate species $\is_{\ell}\neq\is_{\ell'}$. It follows that for any $t>0$ the matrix $\exp(t\La^N)$ is lower triangular. Specifically, $\La^N$ can be written as $\La^N=D^N+T^N$, where $D^N$ is a diagonal matrix and $T^N$ is a lower triangular matrix with zero diagonal entries. Therefore,
  $$\exp(t\La^N)=\exp(tD^N)\exp(tT^N).$$
  The matrix $\exp(tD^N)$ is a diagonal matrix with $\ell$th diagonal entry equal to
  $$e^{t\La^N_{\ell \ell}}=e^{-t\sum_{h\in\ii\cup\fp}\kappa^N_{\ell h}}.$$
  Moreover, since $T^N$ is nilpotent, $\exp(tT^N)$ is a lower triangular matrix whose entries are polynomials in $t\kappa_{\ell \ell'}^N$, for $\ell,\ell'\in\ii$ with $\ell\neq\ell'$. 
  Hence, for any $\varepsilon>0$, any $\ell\in\ii$ such that $y_i\to\is_\ell$ for some $i\in\sr$, and for any $\ell'\in\ii$, $j\in\fp$ such that $\is_{\ell}\rti \is_{\ell'}$ and $\is_{\ell'}\rti y_j$, we have 
  $$N^{\beta^*_\ell-a_j}\, e_{\ell'}^\top \exp\!\pr{N^{a_j-\beta^*_\ell}\varepsilon \La^N} e_{\ell}=N^{\beta^*_\ell-a_j} e^{-N^{a_j-\beta^*_\ell}\varepsilon\sum_{h\in\ii\cup\fp}\kappa^N_{\ell' h}}\exp(N^{a_j-\beta^*_\ell}\varepsilon T^N)_{\ell' \ell}.$$
  The proof is therefore concluded by \eqref{eq:tree} and by the fact that the entries of $\exp(N^{a_j-\beta^*_\ell}\varepsilon T^N)$ are polynomial functions in $N^{a_j-\beta^*_\ell}\varepsilon\kappa_{\ell' \ell''}^N$, which grow at most polynomially in $N$.
 \end{proof}
 
 In addition to the particular cases considered in Propositions \ref{prop:same_N} and \ref{prop:tree}, a simpler sufficient condition implying Assumption \ref{ass:rate_to_zero} is given in the next proposition. A similar result appears in \cite{cappelletti:intermediates}.
  
 \begin{proposition}\label{prop:mu_suff_ass}
  If
  \begin{equation}\label{eq:mu_tend_to_zero}
   \lim_{N\to\infty}N^{\beta^*_\ell-2a_j}\mu_{i\ell'}^N(N^{\alpha}\xni)=0
  \end{equation}
  for all $\xni\in\RR^{|\Sp\setminus\In|}_{\geq0}$, $\ell,\ell'\in\ii$,   $i\in\sr$ and and $j\in\fp$, such that $y_i\to\is_\ell$, $\is_\ell\rti \is_{\ell'}$ and $\is_{\ell'}\rti y_j$, then Assumption \ref{ass:rate_to_zero} holds.
 \end{proposition} 
 The proof of Proposition \ref{prop:mu_suff_ass} is based on the fact that \eqref{eq:mu_tend_to_zero} implies that the integrals of the expression in Assumption \ref{ass:rate_to_zero} tend to zero, see the proof in Section \ref{sec:proofs}.
 Proposition \ref{prop:mu_suff_ass} offers a convenient way to check Assumption \ref{ass:rate_to_zero}, which rarely can be evaluated by hand outside the special cases considered in Propositions \ref{prop:same_N} and \ref{prop:tree}. Even with the use of symbolic computational software, Assumption \ref{ass:rate_to_zero} might be difficult to verify without choosing fixed values for the rate constants. Condition  \eqref{eq:mu_tend_to_zero}  is  in general not a necessary condition, as shown in the next example. The probabilistic formulation of Assumption \ref{ass:rate_to_zero} given in Remark \ref{rem:prob} also often offers an easier way to check Assumption \ref{ass:rate_to_zero}.
 \begin{example}\label{ex:mu_to_infty}
 Consider the following mass-action system
  
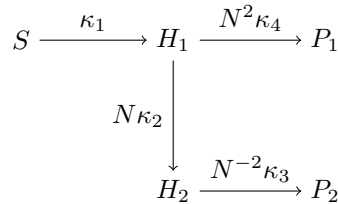
\begin{equation*}
   \begin{tikzpicture}
    \node[state] (S)    at (1,3) {$S$};
    \node[state] (\is1) at (3,3) {$\is_1$};
    \node[state] (\is2) at (3,1) {$\is_2$};
    \node[state] (P_1)  at (5,3) {$P_1$};
    \node[state] (P_2)  at (5,1) {$P_2$};
    \path[->] 
     (S)  edge node{$\kappa_1$} (\is1)
     (\is1) edge node{$N^2\kappa_4$} (P_1)   
            edge node[swap]{$N\kappa_2$}   (\is2)   
     (\is2) edge node{$N^{-2}\kappa_3$} (P_2);   
   \end{tikzpicture}
  \end{equation*}
  If $\alpha_S=\alpha_{P_1}=\alpha_{P_2}=0$, then Assumption \ref{ass:rate_to_zero} is satisfied, but \eqref{eq:mu_tend_to_zero} does not hold.  Indeed, in this case
  $$N^{\beta^*_{\is_1}-2a_{P_2}}\mu_{S,\is_2}^N(N^\alpha \xni)=N\frac{\kappa_1\kappa_2}{N^{-1}\kappa_2\kappa_3+\kappa_3\kappa_4}\xni_S,$$
 which is of order $N$. On the other hand,
  \begin{multline*}
   N^{\beta^*_{\is_1}-a_{P_1}}\, e^\top \exp\pr{N^{a_{P_1}-\beta^*_{\is_1}}\varepsilon \La^N} e_{\is_1}=N^{\beta^*_{\is_1}-a_{P_2}}\, e^\top \exp\pr{N^{a_{P_2}-\beta^*_{\is_1}}\varepsilon \La^N} e_{\is_1}=\\
   \exp\Big(-\varepsilon N(N\kappa_4+\kappa_2)\Big)-\frac{N^3\kappa_4\Big(\exp\Big(-\varepsilon N(N\kappa_4+\kappa_2)\Big)-\exp\Big(-\varepsilon\kappa_3/N^2\Big)\Big)}{N^4\kappa_4+N^3\kappa_2-\kappa_3},
  \end{multline*}
  therefore Assumption \ref{ass:rate_to_zero} holds. Note that in this case, we could have used Proposition \ref{prop:tree} to show that Assumption \ref{ass:rate_to_zero} holds, without calculating the exponential matrix explicitly. The reduced system is
  
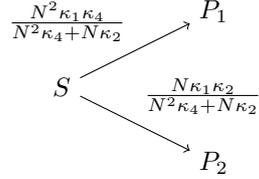
\begin{center}
   \begin{tikzpicture}
    \node[state] (S)    at (1,2) {$S$};
    \node[state] (P_1)  at (3,3) {$P_1$};
    \node[state] (P_2)  at (3,1) {$P_2$};
    \path[->] 
     (S)  edge node{$\frac{N^2\kappa_1\kappa_4}{N^2\kappa_4+N\kappa_2}$} (P_1)
          edge node{$\frac{N\kappa_1\kappa_2}{N^2\kappa_4+N\kappa_2}$}   (P_2);   
   \end{tikzpicture}   
  \end{center}
  and the limiting system is given by
  \begin{center}
   \begin{tikzpicture}
    \node[state] (S)    at (1,2) {$S$};
    \node[state] (P_1)  at (3,2) {$P_1$};
    \path[->] 
     (S)  edge node{$\kappa_1$} (P_1);   
   \end{tikzpicture}   
  \end{center}
 \end{example}
   
    We are now ready to enunciate the following convergence results on compact time intervals:
 \begin{proposition}\label{prop:convergence}
  Let $T>0$ be fixed. Assume  the kinetics $\K^N$ satisfy Assumption \ref{ass:rate_deterministic} for all $N>0$, and that Assumption \ref{ass:single_scale} and \ref{ass:rate_to_zero} hold. Then, if
  $$\lim_{N\to\infty}\|N^{-\alpha}\xni^N(0)-N^{-\alpha}z^N(0)\|=0 \quad\text{and}\quad \xin^N(0)=0,$$
  and if there exists a constant $\Upsilon>0$, such that
  $$\sup_{t\in[0,T], N\in\NN}\|N^{-\alpha}z^N(t)\|<\Upsilon,$$
  we have
  \begin{equation}\label{eq:difference_to_zero}
   \lim_{N\to\infty}\sup_{t\in[0,T]}\|N^{-\alpha}\xni^N(t)-N^{-\alpha}z^N(t)\|=0.
  \end{equation}
 \end{proposition}
 
 \begin{theorem}\label{thm:convergence}
 Let $T>0$ be fixed. Assume  the kinetics $\K^N$ satisfy Assumption \ref{ass:rate_deterministic} for all $N>0$, and that Assumption \ref{ass:single_scale}, \ref{ass:limiting_system} and \ref{ass:rate_to_zero} hold. Then, if
  $$\lim_{N\to\infty}\|N^{-\alpha}\xni^N(0)-z(0)\|=0 \quad\text{and}\quad \xin^N(0)=0$$
  and if there exists a constant $\Upsilon>0$, such that
  $$\sup_{t\in[0,T]}\|z(t)\|<\Upsilon,$$
  we have
   $$\lim_{N\to\infty}\sup_{t\in[0,T]}\|N^{-\alpha}\xni^N(t)-z(t)\|=0.$$ 
 \end{theorem}
 The proofs of the two statements are given in Section \ref{sec:proofs}.
 
 \setcounter{example}{0}
\begin{example}[part 3]
 In this example $\La^N$ is a $1\times 1$ matrix, and it is not difficult to see that Assumption \ref{ass:rate_to_zero} holds if and only if $\alpha_S<\max\{\eta_2,\eta_3\}$. Then,  Proposition \ref{prop:convergence} applies, and we obtain that on a compact interval $[0,T]$, the rescaled concentrations $x^N_E(t)$, $N^{-\alpha_S}x^N_S(t)$ and $N^{-\alpha_P}x^N_P(t)$ are uniformly approximated by $z^N_E(t)$, $N^{-\alpha_S}z^N_S(t)$ and $N^{-\alpha_P}z^N_P(t)$, provided that $N^{-\alpha}(\xni^N(0)-z^N(0))$ goes to zero as $N\to\infty$.  Here $z^N(t)$ is a solution to the reduced reaction system in Example \ref{ex:michaelis_menten} (part 1).
 In this case the limiting reaction system exists, see Example \ref{ex:michaelis_menten} (part 2). Hence  the rescaled trajectory $N^{-\alpha}\xni^N(t)$ can also be approximated in the sense of Theorem \ref{thm:convergence}.
\end{example}

\begin{example}[part 3]   We have
 $$\mu^N_{E+S,\is_1}(x)=\frac{N^4\kappa_1\kappa_3x_Ex_S}{N^6\kappa_3\kappa_4}\quad\text{and}\quad\mu^N_{E+S,\is_2}(x)=\frac{N^3\kappa_1\kappa_2x_Ex_S}{N^6\kappa_3\kappa_4}.$$
 Therefore,  if $\alpha_E=0$ and $\alpha_S<1$, it follows from Proposition \ref{prop:mu_suff_ass}  that Assumption \ref{ass:rate_to_zero} is satisfied. Using a symbolic computational software, it can be verified that Assumption  \ref{ass:rate_to_zero} is satisfied for $\alpha_E=0$ and $\alpha_S<2$, thus for higher values of $\alpha_S$ than given by Proposition \ref{prop:mu_suff_ass}.
 
 In this case the reduced and the limiting reaction systems coincide, so when Assumption \ref{ass:rate_to_zero} holds we can use either Proposition \ref{prop:convergence} or Theorem \ref{thm:convergence} to approximate the trajectories uniformly.
\end{example}

 \section{Discussion}\label{sec:discussion}
 
  \subsection{Long term behaviour}
 
 A natural question arising from Proposition \ref{prop:convergence} and Theorem \ref{thm:convergence} is whether the reduced reaction system or the limiting reaction system also approximates the limit behaviour of the original system as $t\rightarrow\infty$. Specifically, assuming that the limiting system exists, we inquire whether it holds that
 \begin{equation}\label{eq:conv_lim_deterministic}
  \lim_{N\to\infty}\lim_{t\rightarrow\infty}\|N^{-\alpha}\xni^N(t)-z(t)\|\to 0,
 \end{equation}
 when the above limit exist. The answer is that (\ref{eq:conv_lim_deterministic}) may not hold. Consider for example the case where $z(0)$ is an unstable equilibrium point for the reduced reaction network. Then $\lim_{t\rightarrow\infty}z(t)=z(0)$, while in the original reaction network with intermediates, a small perturbation given by the presence of intermediate species may push $N^{-\alpha}\xni^N(t)$ away from the repulsive point $N^{-\alpha}\xni^N(0)=z(0)$. 
   
  Consider the following deterministic mass action system:
  
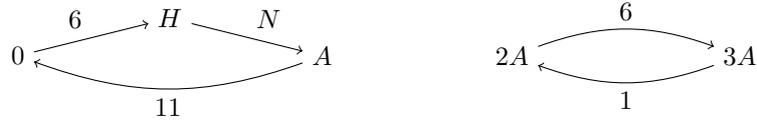
\begin{center}
   \begin{tikzpicture}
    \node[state] (empty)  at (1,1.5)   {$0$};
    \node[state] (A)      at (5,1.5)   {$A$};
    \node[state] (2A)     at (7.5,1.5) {$2A$};
    \node[state] (3A)     at (10.5,1.5) {$3A$};
    \node[state] (\is)      at (3,2) {$\is$};
    \path[->] 
     (A)     edge[bend left] node{$11$} (empty)
     (empty) edge            node{$6$}  (\is)
     (\is)     edge            node{$N$}  (A)
     (2A)    edge[bend left] node{$6$}  (3A)
     (3A)    edge[bend left] node{$1$}  (2A);
   \end{tikzpicture}
  \end{center}
  The assumptions of Theorem \ref{thm:convergence} are fulfilled and the reduced reaction network is given by
  
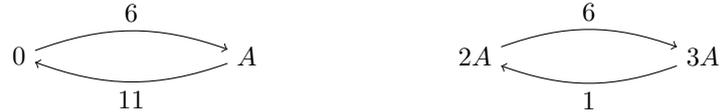
\begin{center}
   \begin{tikzpicture}
    \node[state] (empty)  at (1.5,1.5) {$0$};
    \node[state] (A)      at (4.5,1.5) {$A$};
    \node[state] (2A)     at (7.5,1.5) {$2A$};
    \node[state] (3A)     at (10.5,1.5) {$3A$};
    \path[->] 
     (A)     edge[bend left] node{$11$} (empty)
     (empty) edge[bend left] node{$6$}  (A)
     (2A)    edge[bend left] node{$6$}  (3A)
     (3A)    edge[bend left] node{$1$}  (2A);
   \end{tikzpicture}
  \end{center}
  The ODE governing the dynamics of the reduced reaction network, which does not depend on $N$, is given by
  \begin{align*}
   \frac{d}{dt}z(t)&=-z(t)^3+6z(t)^2-11z(t)+6\\
   &=-\pr{z(t)-1}\pr{z(t)-2}\pr{z(t)-3}\doteqdot f\pr{z(t)}.
  \end{align*}
  Note that 2 is an unstable equilibrium point of the above dynamical system. We will show that if we assume $N^{-\alpha}\xni^N(0)=z(0)=2$ and $\xin^N(0)=0$, then (\ref{eq:conv_lim_deterministic}) does not hold.
  
  The ODE system governing the dynamics of the original network  is
  \begin{align*}
     \dfrac{d}{dt}\xin^N(t)&=6-N\xin^N(t)\\ 
     \dfrac{d}{dt}\xni^N_A(t)&=-\xni^N(t)^3+6\xni^N(t)^2-11\xni^N(t)+N\xin^N(t).
    \end{align*}
  This means that $\xin^N(t)=6(1-e^{-Nt})/N$ and 
  $$\frac{d}{dt}\xni^N(t)=-\xni^N(t)^3+6\xni^N(t)^2-11\xni^N(t)+6\pr{1-e^{-Nt}}\doteqdot g^N_t\pr{\xni^N(t)}$$
  Since for any $N\geq1$ and $t>0$, we have  $g^N_t(x)<f(x)$, and since $f(x)<0$ for any $x\in\pr{1,2}$, then
  $$\lim_{t\rightarrow\infty}\xni^N(t)\leq 1.$$
  It is possible to prove a more precise result, namely that $\lim_{t\rightarrow\infty}\xni^N(t)= 1$, and it is worth noting that 1 is a stable steady state of the reduced reaction system. Since
  $$\lim_{t\rightarrow\infty}z(t)=z(0)=2,$$
  we have that \eqref{eq:conv_lim_deterministic} does not hold.
  
  Few questions are, however, left open by this counterexample. First of all, if instead of $\xin^N(0)=0$ we had $\xin^N(0)=6/N$, then we would have $\xin^N(t)=6/N$ for any $t\in\RR_{\geq0}$, and the dynamics of the non-intermediate species in the full and in the reduced reaction network would coincide for all $t\ge 0$. This is true in general, whenever we can impose $d \xin^N(t)/dt$ to be  0 for all time. This would imply that the dynamics  of the system is confined within the so called \emph{slow manifold}. It is natural to wonder when this is possible.  Another natural question is whether \eqref{eq:conv_lim_deterministic} is false only if some instability of the system is present, as in the previous case. In general, what conditions could assure \eqref{eq:conv_lim_deterministic}?

  \subsection{Some relationships to other approaches}
  
It has previously been demonstrated that under certain conditions, an ODE system with fast and slow reactions (two categories only) might be transformed into an equivalent ODE system with fast and slow variables for which Tikhonov's approach is applicable \cite{schauer-heinrich,goeke-walcher}. In general, the ODE systems we consider do not fulfil the requirements for this transformation to be valid. It is, however, worth pointing out that a standard assumption in Tikhonov's approach, namely that the eigenvalues of the Jacobian of the fast subsystem have negative real parts, is also fulfilled in our case (here the Jacobian corresponds to the matrix $L^N$ in \eqref{eq:laplacian}).

Reference \cite{radulescu:linear} also provides a method to reduced a multiscale reaction system to a smaller reaction system. Reactions are removed iteratively   in such a way that only the reactions with the highest rates remain. If a reaction system can be reduced by our method as well as by their method, the two reduced reaction systems might not agree; as illustrated by the following example:
    
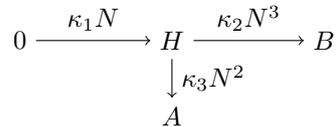
\begin{equation*}
   \begin{tikzpicture}[baseline={(current bounding box.center)}]
   \node[state] (0)   at (0,2) {$0$};
   \node[state] (\is) at (2,2) {$\is$};
   \node[state] (A)   at (2,1) {$A$};
   \node[state] (B)  at (4,2) {$B$};
   \path[->] 
    (0)   edge node{$\kappa_1N$}   (\is)
    (\is) edge node{$\kappa_2N^3$} (B)
    (\is) edge node{$\kappa_3N^2$}   (A);
   \end{tikzpicture}
  \end{equation*}
  Let $x^N(t)$ denote the solution to the system for some initial condition $x^N(0)$. Here, the concentration of the species $A$ grows with rate of order $O(1)$, while the concentration of the species $B$ changes with rate of order $O(N)$. In accordance with these rates, we assume that the abundance of the species $A$ and $B$ are such that $x^N_A(0)=O(1)$ and $x^N_B(0)=O(N)$.
 
    In \cite{radulescu:linear}, the reduction is performed such that only the reaction consuming $\is$ with the highest rate is kept, and this leads to the following reduced reaction system
  
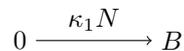
\begin{equation*}
   \begin{tikzpicture}[baseline={(current bounding box.center)}]
   \node[state] (0)   at (0,2) {$0$};
   \node[state] (B)  at (2,2) {$B$};
   \path[->] 
    (0)   edge node{$\kappa_1N$}   (B);
   \end{tikzpicture}
  \end{equation*}
  where the concentration of $A$ remains constant. However, in the original system, the concentration of  $A$ grows at rate  $O(1)$, so its dynamics is not well captured by the reduced model. Note that the concentration of species $A$ could be important for the dynamics of another part of the network, for example if $A$ is an enzyme catalysing a reaction of interest. If that is the case, using the above reduced system could lead to an important error.
  
  The reduced system  we propose, which correctly approximate the dynamics of the original system by Theorem \ref{thm:convergence}, is the following:
  
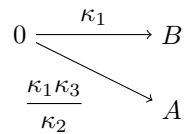
\begin{equation*}
   \begin{tikzpicture}[baseline={(current bounding box.center)}]
   \node[state] (0)   at (0,2) {$0$};
   \node[state] (A)   at (2,1) {$A$};
   \node[state] (2B)  at (2,2) {$B$};
   \path[->] 
    (0)   edge node{$\kappa_1$}    (B)
    (0) edge node[swap]{$\dfrac{\kappa_1\kappa_3}{\kappa_2}$} (A);
   \end{tikzpicture}
  \end{equation*}
  Denote by $z(t)$ the solution of the latter, and suppose $z(0)=\lim_{N\to\infty}x_A^N(0)N^{-1}x_B^N(0)$. Then, by Theorem \ref{thm:convergence}, on compact time intervals, $z_A(t)$ and $z_B(t)$ provide a uniform limit for $x^N_A(t)$ and $N^{-1}x^N_B(t)$, respectively.

  \section{Proofs}\label{sec:proofs}
  
 This section contains  the proofs of Lemma \ref{lem:mu_and_inverse}, Proposition \ref{prop:mu_suff_ass}, Proposition \ref{prop:convergence} and Theorem \ref{thm:convergence}. 
 
 \subsection*{Proof of Lemma \ref{lem:mu_and_inverse}}
  The result does not depend on $N$, thus for the sake of simplicity $N$ is suppressed in the notation of this proof. Consider the Laplacian matrix in \eqref{eq:laplacian}. The first $|\In|$ columns (and rows) are indexed by $\ii$, and let $q$ be the index of the last column (row). By the matrix tree theorem \cite{tutte:matrix_tree} we have
  $$\mu_{i\ell}\pr{\xni}=\frac{\sum_{\sigma\in\theta_{i,\xni}\pr{\is_\ell}}w(\sigma)}{\sum_{\sigma\in\theta_{i,\xni}\pr{\star}}w(\sigma)}
  =\frac{\det \pr{\La_i^{\xni}}_{\pr{\ell,\ell}}}{\det \La},$$
  where  $\pr{\La_i^{\xni}}_{\pr{\ell,\ell'}}$ are the minors of $\La_i^{\xni}$. Since the last row of $\La_i^{\xni}$ is minus the sum of the other rows, we have
  $$
   \det \pr{\La_i^{\xni}}_{\pr{\ell,\ell}}=(-1)^{\ell+|\In| +1}\det \pr{\La_i^{\xni}}_{\pr{q,\ell}}=-\det\pr{\La_i^{\xni}}_{\pr{q,\widehat{\ell}\,}},
  $$
  where $\pr{\La_i^{\xni}}_{\pr{q,\widehat{\ell}\,}}$ denotes the matrix $\La_i^{\xni}$ with the last row eliminated, the column indexed by $\ell$ replaced by the column $\lambda_i(\xni)$, and the last column eliminated. The last equality follows from changing the order of the columns. Moreover,  by Cramer's Rule we have
  $$\mu_{i\ell}\pr{\xni}=\frac{-\det\pr{\La_i^{\xni}}_{\pr{q,\widehat{\ell}\,}}}{\det \pr{\La_i^{\xni}}_{\pr{q,q}}}=-\pr{\La^{-1}\lambda_i(\xni)}_{\ell},$$
  which concludes the proof.\qed 

  \subsection*{Preliminary results}
 
  Before proving Proposition \ref{prop:mu_suff_ass}, Proposition \ref{prop:convergence} and Theorem \ref{thm:convergence}, we need some preliminary results. In order to prove Proposition \ref{prop:mu_suff_ass}, only the first lemma is necessary, which is concerned with some properties of the matrix $\exp\pr{\La^N t}$.
 
 \begin{lemma}\label{lem:exp_non_negative}
  The following statements are true.
  \begin{enumerate}[i)]
   \item\label{item:exp_non_negative} For any $t>0$, any entry of the matrix $\exp\pr{\La^N t}$ is non-negative. 
   \item\label{item:exp_to_zero} We have
   $$\lim_{t\to\infty}\exp\pr{\La^N t}=0.$$
   \item\label{item:exp_monotone} For $0\leq s\leq t$, we have that
   $$e^\top\exp\pr{\La^N s}\geq e^\top\exp\pr{\La^N t}.$$
   In particular, for $s=0$ and any $t>0$, $e^\top\exp\pr{\La^N t}\leq e^\top$.
   \item\label{item:inverse_non_negative} Any entry of the matrix $-(\La^N)^{-1}$ is non-negative.
   \item\label{item:exp_monotone_with_inverse} For any non-negative vector $\xin\in\R_{\geq0}^{|\In|}$ and  any $0\leq s\leq t$,
   $$0\leq-\vkappaj^N(\La^N)^{-1}\exp(\La^N t)\xin\leq -\vkappaj^N(\La^N)^{-1}\exp(\La^N s)\xin.$$
   In particular, for $s=0$ and any $t>0$,
   $$0\leq-\vkappaj^N(\La^N)^{-1}\exp(\La^N t)\xin\leq -\vkappaj^N(\La^N)^{-1}\xin.$$
  \end{enumerate}
 \end{lemma}
 
 \begin{proof}
  If we put $\Lambda^N(t)\equiv 0$, then from (\ref{eq:solution_for_\is}) we have
  $$\xin^N(t)=\exp\pr{\La^N t}\xin^N(0).$$
  This implies that each column of $\exp\pr{\La^N t}$ represents the concentrations of the intermediate species at time $t$ given the initial condition $\xin^N(0)=e_{\ell}$. In turn this implies that the entries of $\exp\pr{\La^N t}$ must be non-negative for any $t>0$, which proves part \eqref{item:exp_non_negative}. Furthermore, the condition $\Lambda^N(s)\equiv 0$ implies that the intermediates are not produced, thus the sum of their concentrations decreases independently on their actual value. Indeed, the stoichiometric coefficients of the intermediate species are one, hence the net flow among intermediates is 0, while they can degrade to produce a non-intermediate complex. These considerations prove parts \eqref{item:exp_to_zero} and \eqref{item:exp_monotone}. Finally, $-\La^N$ is a Z-matrix, as all non-diagonal entries are non-positive. Using  the first and the third Gershgorin Theorems \cite{book:num_an} as in Remark \ref{rem:L_not_singular}, it can be shown that the real parts of the eigenvalues of $-\La^N$ are strictly positive. Namely, for any $N\in\NN$, $-\La^N$ is a non-singular M-matrix and in particular all the entries of $-(\La^N)^{-1}$ are non-negative \cite{book:matrices}. The proof of part \eqref{item:inverse_non_negative} is therefore concluded. For part \eqref{item:exp_monotone_with_inverse}, the former inequality follows from parts \eqref{item:exp_non_negative} and \eqref{item:inverse_non_negative}. The latter inequality follows from
  \begin{align*}
   -\vkappaj^N(\La^N)^{-1}\exp(\La^N t)\xin&=\int_t^\infty \vkappaj^N\exp(\La^N u)\xin du\\
   \leq&\int_s^\infty \vkappaj^N\exp(\La^N u)\xin du\\
   =&-\vkappaj^N(\La^N)^{-1}\exp(\La^N s)\xin,
  \end{align*}
  where the equalities in the first and the third lines follow from part \eqref{item:exp_to_zero}, while the inequality in the second line follows from part \eqref{item:exp_non_negative}. 
 \end{proof}

 \begin{lemma}\label{lem:appendix1}
  Consider the notation introduced in Sections \ref{sec:multiscale} and \ref{sec:reduced}. Let $T>0$ be fixed. Assume that the kinetics $\K^N$, $N\in\NN$, satisfy Assumption \ref{ass:rate_deterministic}. Furthermore, assume that Assumption \ref{ass:rate_to_zero} holds, and that there exists a constant $\Upsilon>0$, such that
  \begin{equation}\label{eq:bound_two_solutions}
  \sup_{\substack{t\in[0,T]\\ N\in\NN}} \|N^{-\alpha}\xni^N(t)\|+\|N^{-\alpha}z^N(t)\| <\Upsilon.
  \end{equation}
  Finally, assume that
  $$\lim_{N\to\infty}\|N^{-\alpha}\xni^N(0)-N^{-\alpha}z^N(0)\|=0 \quad\text{and}\quad \xin^N(0)=0.$$
  Then, for any $y_j\in\C$,
  \begin{equation*}\label{eq:lemma_to_be_proved}
  \lim_{N\to\infty}\sup_{t\in[0,T]}\,-N^{-\alpha}y_{j}\vkappaj^N\int_0^t (\La^N)^{-1}\exp(\La^N(t-s))\Lambda^N(s)ds=0.
  \end{equation*}
 \end{lemma}
 \begin{proof}
 By linearity, what we need to prove is that for every $k$ such that $y_{jk}\neq0$  
  \begin{equation}\label{eq:to_be_proved_1}
   \lim_{N\to\infty}\sup_{t\in[0,T]}\,-N^{-\alpha_k}\vkappaj^N\int_0^t (\La^N)^{-1}\exp(\La^N(t-s))e_\ell\Lambda_\ell^N(s)ds=0.
  \end{equation}
  Moreover, by standard properties of the Laplacian matrix of a graph we have that the entries
  $$(\exp(L^Nt)e_\ell)_{\ell'}\quad\text{and}\quad \pr{(L^N)^{-1}e_\ell}_{\ell'}$$
  are different from zero if and only if $\is_\ell\rti \is_{\ell'}$. For any $j\in\fp$, let $\ii_j\subseteq\ii$ be the set of indices $\ell$ for which $\is_{\ell}\rti y_j$. Note that $\is_{\ell}\rti y_j$ if and only if either $\kappa^N_{\ell j}>0$ or $\is_\ell\rti \is_{\ell'}$ with $\kappa^N_{\ell'j}>0$. Therefore, it suffices to show \eqref{eq:to_be_proved_1} for $k$ such that $y_{jk}\neq0$ and for any  $\ell\in\ii_j$. Moreover, since $\Lambda_\ell^N(s)$ is different form zero only if $y_i\to\is_\ell\in\R$ for some $i\in\sr$, it suffices to prove \eqref{eq:to_be_proved_1} only for those $\ell$ such that $y_i\to\is_\ell\in\R$, for some $i\in\sr$.
  
  By \eqref{eq:multiscale_beta} and \eqref{eq:bound_two_solutions}, there exists a positive constant $B_\ell$ such that
  $$\sup_{\substack{N\in\NN \\ t\in[0,T]}}N^{-\beta^*_\ell}\Lambda_\ell^N(t)<B_\ell.$$
  Therefore, since $-\vkappaj^N(\La^N)^{-1}\exp(\La^N(t-s))$ has all non-negative entries according to Lemma \ref{lem:exp_non_negative}\eqref{item:exp_monotone_with_inverse}, in order to prove \eqref{eq:to_be_proved_1} it suffices to show that
  \begin{equation}\label{eq:to_be_proved_2}
   \lim_{N\to\infty}\sup_{t\in[0,T]}\,-N^{\beta^*_\ell-\alpha_k}\vkappaj^N\int_0^t (\La^N)^{-1}\exp(\La^N(t-s))e_\ell ds=0.
  \end{equation}
  By definition of $\La^N$, we have
  \begin{equation*}
   \kappa^N_{\ell j}\leq\sum_{j\in\fp}\kappa^N_{\ell j}= -\pr{e^\top \La^N}_\ell
  \end{equation*}
  By Lemma \ref{lem:exp_non_negative}\eqref{item:inverse_non_negative} it follows $-\kappa^N_{\cdot j}(\La^N)^{-1}\leq e^\top$. Moreover, if $\ell'\rti y_j$ does not hold, 
  $$\Big(\kappa^N_{\cdot j}(\La^N)^{-1}\Big)_{\ell'}=\sum_{\ell\in\ii}\kappa^N_{\ell j}(\La^N)^{-1}_{\ell \ell'}=0.$$
  Indeed, if it were $\kappa^N_{\ell j}(\La^N)^{-1}_{\ell \ell'}\neq0$ for some $\ell\in\ii$, it would follow $\is_{\ell'}\rti\is_{\ell}$ and $\is_{\ell}\to y_j$, which would in turn imply $\ell'\rti y_j$. Hence, we can conclude
  \begin{equation}\label{eq:k_leq_La^N_new}
   -\kappa^N_{\cdot j}(\La^N)^{-1}\leq \sum_{\ell'\in\ii\,:\, \is_{\ell'}\rti y_j} e^\top_{\ell'} \leq e^\top.
  \end{equation}
  
 We distinguish between two different cases: first, suppose that $\alpha_k>\beta^*_\ell$. Then, we have by application of the first part of Lemma \ref{lem:exp_non_negative}\eqref{item:exp_monotone_with_inverse} for the first inequality and by application of the second part of Lemma \ref{lem:exp_non_negative}\eqref{item:exp_monotone_with_inverse} together with \eqref{eq:k_leq_La^N_new} for the second inequality, that
  $$0\leq \sup_{t\in[0,T]}-N^{\beta^*_\ell-\alpha_k}\vkappaj^N\int_0^t (\La^N)^{-1}\exp(\La^N(t-s))e_\ell ds\leq \sup_{t\in[0,T]}N^{\beta^*_\ell-\alpha_k}\int_0^t e^\top e_\ell ds\leq N^{\beta^*_\ell-\alpha_k}T$$
  and the latter tends to 0 as $N$ tends to infinity, proving \eqref{eq:to_be_proved_2}. If $\alpha_k\leq\beta^*_\ell$, then, due to Lemma \ref{lem:exp_non_negative}\eqref{item:exp_monotone_with_inverse} for the first inequality below, \eqref{eq:k_leq_La^N_new} and Lemma \ref{lem:exp_non_negative}\eqref{item:exp_non_negative} for the second inequality, Lemma \ref{lem:exp_non_negative}\eqref{item:exp_monotone_with_inverse} and \eqref{item:exp_monotone} for the third, and \eqref{eq:k_leq_La^N_new} and Lemma \ref{lem:exp_non_negative}\eqref{item:exp_non_negative} for the forth, we have for any $\varepsilon>0$, 
  \begin{align*}
   0\leq \sup_{t\in[0,T]}-N^{\beta^*_\ell-\alpha_k}\vkappaj^N\int_0^t (\La^N)^{-1}\exp(\La^N(t-s))e_\ell \,ds &\\
   &\hspace*{-119pt}\leq \sup_{t\in[0,T]}-N^{\beta^*_\ell-\alpha_k}\vkappaj^N(\La^N)^{-1}\int_0^{(t-\varepsilon N^{\alpha_k-\beta^*_\ell})\vee 0} \exp(\La^N(t-s))e_\ell \,ds\\
   &\hspace*{-119pt}\quad+\sup_{t\in[0,T]} N^{\beta^*_\ell-\alpha_k}\int_{(t-\varepsilon N^{\alpha_k-\beta^*_\ell})\vee 0}^t e^\top\exp(\La^N(t-s))e_\ell \,ds\\
   &\hspace*{-119pt}\leq N^{\beta^*_\ell-\alpha_k}\pr{-T\vkappaj^N(\La^N)^{-1}\exp(\La^N\varepsilon N^{\alpha_k-\beta^*_\ell})e_\ell+ \varepsilon N^{\alpha_k-\beta^*_\ell} e^\top e_\ell } \\
   &\hspace*{-119pt}\leq N^{\beta^*_\ell-\alpha_k}T\sum_{\ell'\in\ii\,:\, \is_{\ell'}\rti y_j} e^\top_{\ell'}\exp(\La^N\varepsilon N^{\alpha_k-\beta^*_\ell})e_\ell+\varepsilon.
  \end{align*}
  By Assumption \ref{ass:rate_to_zero} the latter tends to $\varepsilon$ as $N$ tends to infinity, and the proof is concluded by the arbitrariness of $\varepsilon>0$.
 \end{proof}

 We are now ready for the proof of Proposition \ref{prop:convergence} and Theorem \ref{thm:convergence}.
 
 \subsection*{Proof of Proposition \ref{prop:mu_suff_ass}}

 Note that, due to Lemma \ref{lem:exp_non_negative}\eqref{item:exp_to_zero},
 \begin{equation}\label{eq:expectation_to_go_to_zero}
 N^{\beta^*_\ell-a_j}\int_0^\infty e_{\ell'}^\top\exp\pr{N^{a_j-\beta^*_\ell} \La^Ns}e_\ell ds=-N^{2\beta^*_\ell-2a_j}e_{\ell'}^\top(\La^N)^{-1}e_\ell, 
 \end{equation}
 for any $\ell,\ell'\in\ii$. By \eqref{eq:multiscale_beta}, for any $i\in\sr$ and $\ell\in\ii$ with $y_i\to \is_\ell\in\R$, $\lambda^N_{i\ell}(N^{\alpha}\xni)=O(N^{\beta_{i \ell}})$ for any $\xni\in\RR^{|\Sp\setminus\In|}$, and there exists $\xni\in\RR^{|\Sp\setminus\In|}$ such that $\lambda^N_{i\ell}(N^{\alpha}\xni)=\Theta(N^{\beta_{i \ell}})$. Therefore, \eqref{eq:expectation_to_go_to_zero} goes to zero as $N$ tends to infinity if and only if
 $$\lim_{N\to\infty}-N^{\beta^*_\ell-2a_j}e_{\ell'}^\top(\La^N)^{-1}e_\ell\lambda^N_{i\ell}(N^{\alpha}\xni)=0,$$
 for any $i\in\sr$ with $y_i\to\is_{\ell}$ and any $\xni\in\RR^{|\Sp\setminus\In|}$. By Lemma \ref{lem:mu_and_inverse}, the latter holds for any $i\in\sr$, $\ell, \ell'\in\ii$ and $j\in\fp$ such that $y_i\to \is_{\ell}$, $\is_{\ell}\rti\is_{\ell'}$ and $\is_{\ell'}\rti y_j$, if and only if
 $$\lim_{N\to\infty}N^{\beta^*_\ell-2a_j}\mu_{i\ell'}^N(N^{\alpha}\xni)=0,$$
 for any $i\in\sr$, $\ell, \ell'\in\ii$ and $j\in\fp$ such that $y_i\to \is_{\ell}$, $\is_{\ell}\rti\is_{\ell'}$ and $\is_{\ell'}\rti y_j$. The latter implies that
 $$\lim_{N\to\infty}N^{\beta^*_\ell-a_j}\int_0^\infty e_{\ell'}^\top\exp\pr{N^{a_j-\beta^*_\ell} \La^Ns}e_\ell ds=0$$
 for any $i\in\sr$, $\ell, \ell'\in\ii$ and $j\in\fp$ such that $y_i\to \is_{\ell}$, $\is_{\ell}\rti\is_{\ell'}$ and $\is_{\ell'}\rti y_j$.
 This in turn implies that Assumption \ref{ass:rate_to_zero} holds, since the entries of $e_{\ell'}^\top\exp\pr{N^{a_j-\beta^*_\ell} \La^Ns}$ are non-negative and non-increasing by Lemma \ref{lem:exp_non_negative}\eqref{item:exp_non_negative} and \eqref{item:exp_monotone}.

 \subsection*{Proof of Proposition \ref{prop:convergence}}
 
  We first assume that \eqref{eq:bound_two_solutions} holds, that is, there exists a constant $\Upsilon>0$ such that
  \begin{equation*}
  \sup_{\substack{t\in[0,T]\\ N\in\NN}}\,\|N^{-\alpha}\xni^N(t)\|+\|N^{-\alpha}z^N(t)\|<\Upsilon.\tag{\ref{eq:bound_two_solutions}}
  \end{equation*}
  We will drop this assumption later. For convenience, we introduce the vector $\overline{\Lambda}^N(t)$ of length $|\In|$ with  entries  indexed by $\In$ and
  $$\overline{\Lambda}^N_\ell(t)=\sum_{i\in\sr}\lambda^N_{i\ell}(z(t)).$$  
  Due to \eqref{eq:delayed_ODE} and Lemma \ref{lem:mu_and_inverse}, we have
  \begin{equation}\label{eq:break}
   \|N^{-\alpha}\xni^N(t)-N^{-\alpha}z^N(t)\|\leq\|N^{-\alpha}\xni^N(0)-N^{-\alpha}z^N(0)\|+ \|A^N(t)\|+\|B^N(t)\|,
  \end{equation}
  where
  \begin{align*}
   A^N(t)=&N^{-\alpha}\int_0^t\pr{\sum_{j\in\fp}y_j\vkappaj^N\int_0^u\exp\pr{\La^N (u-s)}\Lambda^N(s) ds+\sum_{j\in\fp}y_j\vkappaj^N(\La^N)^{-1}\Lambda^N(u)}du\\
 \intertext{and}
   B^N(t)=&N^{-\alpha}\int_0^t\pr{\sum_{j\in\fp}y_j\vkappaj^N(\La^N)^{-1}\overline{\Lambda}^N(u)-\sum_{j\in\fp}y_j\vkappaj^N(\La^N)^{-1}\Lambda^N(u)}du\\
   &+N^{-\alpha}\int_0^t\pr{\sum_{\substack{i\notin\ii \\ 1\leq j\leq |\C|}}\pi(y_j-y_i)\pr{\lambda^N_{ij}(\xni(t))-\lambda^N_{ij}(z(t))}}du.
  \end{align*}
  We have
  \begin{align*}
   A^N(t)=&\sum_{j\in\fp}N^{-\alpha}y_j\vkappaj^N\left[\int_0^t\pr{\int_s^t\exp\pr{\La^N (u-s)}du}\Lambda^N(s) ds +(\La^N)^{-1}\int_0^t \Lambda^N(s)ds \right]\\
   =&\sum_{j\in\fp}N^{-\alpha}y_j\vkappaj^N\left[\int_0^t(\La^N)^{-1}\pr{\exp\pr{\La^N (t-s)}-I}\Lambda^N(s)ds +(\La^N)^{-1}\int_0^t \Lambda^N(s)ds \right]\\
   =&\sum_{j\in\fp}N^{-\alpha}y_j\vkappaj^N\int_0^t(\La^N)^{-1}\exp\pr{\La^N (t-s)}\Lambda^N(s)ds.
  \end{align*} 
 Therefore, by Lemma \ref{lem:appendix1} we have
  \begin{equation}\label{eq:sup_A_to_zero}
  \lim_{N\to\infty}\sup_{t\in[0,T]}\|A^N(t)\|=0. 
  \end{equation}  
  Moreover, it follows from 
  $$\sum_{j\in\fp}\vkappaj^N(\La^N)^{-1}=e^\top\La^N (\La^N)^{-1}=e^\top$$
 (by definition of $\La^N$) or from \eqref{eq:sum}, that $B(t)$ can be written as
  \begin{align*}
   B^N(t)=&\int_0^t\sum_{\substack{i\in\sr\\ j\in\fp}}N^{-\alpha}(y_j-y_i)\pr{\sum_{\ell\in\ii}\vkappaj^N(\La^N)^{-1}e_\ell\Big(\lambda^N_{i\ell}(\xni^N(u))-\lambda^N_{i\ell}(z^N(u))\Big)}du\\
   &+\int_0^t\sum_{i,j\notin\ii}N^{-\alpha}(y_j-y_i)\pr{\lambda^N_{ij}(\xni^N(u))-\lambda^N_{ij}(z^N(u))}du.
  \end{align*}
  In particular, due to \eqref{eq:bound_two_solutions} and \eqref{eq:multiscale_beta}, for any $\varepsilon>0$ and any $N$ large enough
  \begin{align*}
   \|B^N(t)\|\leq&\int_0^t\sum_{\substack{i\in\sr\\ j\in\fp}}\|N^{-\alpha}(y_j-y_i)\|\pr{\sum_{\ell\in\ii}\vkappaj^N(\La^N)^{-1}e_\ell N^{\beta_{i\ell}}\Big\|\lambda_{i\ell}(N^{-\alpha}\xni^N(u))-\lambda_{i\ell}(N^{-\alpha}z^N(u))+\varepsilon\Big\|}du\\
   &+\int_0^t\sum_{i,j\notin\ii}N^{\beta_{ij}}\Big\|N^{-\alpha}(y_j-y_i)\pr{\lambda_{ij}(N^{-\alpha}\xni^N(u))-\lambda_{ij}(N^{-\alpha}z^N(u))+\varepsilon}\Big\|du.
  \end{align*}
  Note that the limit functions $\lambda_{ij}$ are locally Lipschitz, which implies that they are Lipschitz on compact sets contained in their domain. Moreover, from Assumption \ref{ass:single_scale} and  Remark \ref{rem:equivalent}, it follows that there exists two positive constants $0<\Gamma_0,\Gamma_1<\infty$ such that
  $$\|B^N(t)\|\leq \Gamma_0\varepsilon + \Gamma_1\int_0^t\|N^{-\alpha}\xni^N(u)-N^{-\alpha}z^N(u)\|du.$$
  Hence, by \eqref{eq:break} and Gronwall inequality, 
  \begin{align*}
   \|N^{-\alpha}\xni^N(t)-N^{-\alpha}z^N(t)\|\leq& \|N^{-\alpha}\xni^N(0)-N^{-\alpha}z^N(0)\|+\|A^N(t)\|+\Gamma_0\varepsilon\\
   &+\exp(\Gamma_1 t)\int_0^t \pr{\|N^{-\alpha}\xni^N(0)-N^{-\alpha}z^N(0)\|+\|A^N(s)\|+\Gamma_0\varepsilon}ds.
  \end{align*}
  Hence, \eqref{eq:difference_to_zero} follows from \eqref{eq:sup_A_to_zero}, from the arbitrariness of $\varepsilon>0$ and from the hypothesis
  $$\lim_{N\to\infty}\|N^{-\alpha}\xni^N(0)-N^{-\alpha}z^N(0)\|=0.$$
  
  To complete the proof, we need to prove \eqref{eq:difference_to_zero} without assuming \eqref{eq:bound_two_solutions}. We will do so by showing that \eqref{eq:bound_two_solutions}  follows from what we have already shown. By hypothesis we have that there exists a finite positive constant $\Upsilon$ such that 
  \begin{equation}\label{eq:zeta_bounded}
   \sup_{\substack{t\in[0,T]\\ N\in\NN}}\|N^{-\alpha}z^N(t)\|<\Upsilon.
  \end{equation}
We need to prove the existence of an upper bound for the rescaled solutions $N^{-\alpha}x^N(t)$. Choose a constant $0<\delta<1$ and consider the following modified kinetics: for any $1\leq i,j\leq |\C|$ with $y_i\to y_j\in\overline{\R}$, we let
  $$\widetilde{\lambda}^N_{ij}(\xni)=\begin{cases}
                                                            \lambda^N_{ij}(\xni)&\text{if }\|N^{-\alpha}\xni\|\leq \Upsilon+\delta\\
                                                            (1+\|N^{-\alpha}\xni\|-\Upsilon-\delta)\lambda^N_{ij}\pr{\frac{\Upsilon+\delta}{\|N^{-\alpha}\xni\|}\xni}+(\|N^{-\alpha}\xni\|-\Upsilon-\delta)N^{\beta_{ij}}&
                        \text{if }\Upsilon+\delta<\|N^{-\alpha}\xni\|\leq \Upsilon+\delta+1\\
                        N^{\beta_{ij}}&\text{otherwise.}   
                                                           \end{cases}$$
  The consumption rates of intermediate species are not modified, as well as the starting conditions. Let $\widetilde{x}^N(t)$ be the projection  of the solution of the modified system onto the space of the non-intermediate species. The modified kinetics for $(\Sp,\C, \R)$ lead to a new family of kinetics for the reduced reaction network $(\Sp\setminus\In,\C\setminus\In,\Rr)$, which is defined as in \eqref{eq:reduced_rates_deterministic}. The solution to the modified reduced reaction systems, however, are still $z^N(t)$, assuming the same initial condition. Indeed, due to \eqref{eq:zeta_bounded}, the argument of the reaction rates of the modified  kinetics have always norm smaller than $\Upsilon+\delta$, so the changes in the kinetics have no effect.
  On the other hand, as  the convergence in \eqref{eq:multiscale_beta}  is uniform on compact sets, the modified reaction rates are such that $N^{-\beta_ij}\widetilde{\lambda}^N_{ij}(N^{\alpha}\xni)$ converges uniformly to
  $$\widetilde{\lambda}_{ij}(\xni)=\begin{cases}
                                                            \lambda_{ij}(\xni)&\text{if }\|\xni\|\leq \Upsilon+\delta\\
                                                            (1+\|\xni\|-\Upsilon-\delta)\lambda_{ij}\pr{\frac{\Upsilon+\delta}{\|\xni\|}\xni}+(\|\xni\|-\Upsilon-\delta)&
                        \text{if }\Upsilon+\delta<\|\xni\|\leq \Upsilon+\delta+1\\
                        1&\text{otherwise.}   
                                                           \end{cases}$$
  Note that the limit functions $\widetilde{\lambda}_{ij}(\xni)$ are bounded. This implies that the functions $N^{-\beta_ij}\widetilde{\lambda}^N_{ij}(\xni)$ are uniformly bounded by a finite positive constant $\Upsilon'$. Therefore, by \eqref{eq:delayed_ODE}, \eqref{eq:sum} and Lemmas \ref{lem:mu_and_inverse} and \ref{lem:exp_non_negative},  
  $$
   \frac{d}{dt}N^{-\alpha}\widetilde{x}^N(t)\leq \sum_{\substack{i\in\sr\\ j\in\fp}}N^{-\alpha}(y_j-y_i)\sum_{\ell\in\ii}\vkappaj^N\int_0^t\exp\pr{\La (t-s)}e_\ell N^{\beta_{i\ell}}\Upsilon' ds + \sum_{i,j\notin\ii}N^{-\alpha}(y_j-y_i)N^{\beta_{ij}}\Upsilon'.
  $$
 By Assumption \ref{ass:single_scale} and  Remark \ref{rem:equivalent}, we conclude that the derivative $\frac{d}{dt}N^{-\alpha}\widetilde{x}^N(t)$ is uniformly bounded for $t\in[0,T]$ and $N\in\NN$. It follows that there exists a finite constant $\Upsilon''$ such that
 $$\sup_{\substack{t\in[0,T]\\ N\in\NN}}\|N^{-\alpha}\widetilde{x}^N(t)\|<\Upsilon'',$$
 which implies that the solutions $\widetilde{x}^N(t)$ are uniformly bounded for $t\in[0,T]$ and $N\in\NN$. Then, for what we have shown in the first part of the proof,
 $$\lim_{N\to\infty}\sup_{t\in[0,T]}\|N^{-\alpha}\widetilde{x}^N(t)-N^{-\alpha}z^N(t)\|=0.$$
 In particular, this means that for $N$ large enough
 $$\sup_{t\in[0,T]}\|N^{-\alpha}\widetilde{x}^N(t)\|<\Upsilon+\delta.$$
 Hence, for $N$ large enough the modification of the kinetics does not affect the solutions $\widetilde{x}^N(t)$, for $t\in[0,T]$. Therefore, for $N$ large enough we have that for any $t\in[0,T]$
 $$\widetilde{x}^N(t)=\xni^N(t).$$
 It follows that \eqref{eq:bound_two_solutions} holds, by eventually changing $\Upsilon$ to $2\Upsilon+\delta$, and this concludes the proof. \qed
 
 \subsection*{Proof of Theorem \ref{thm:convergence}}

By hypothesis, there exists a finite positive constant $\Upsilon$ such that
\begin{equation}\label{eq:limit_zeta_bounded}
 \sup_{t\in[0,T]}\|z(t)\|<\Upsilon.
\end{equation}
Following the same trick used in the proof of Proposition \ref{prop:convergence}, we consider a modification of the kinetics \eqref{eq:reduced_rates_deterministic} for the reduced reaction network $(\Sp\setminus\In, \C\setminus\In,\Rr)$. Choose a constant $0<\delta<1$ and for each $y_i\to y_j\in\Rr$, define 
  $$\mlambdarN_{ij}(\xni)=\begin{cases}
                                                            \lambdarN_{ij}(\xni)&\text{if }\|N^{-\alpha}\xni\|\leq \Upsilon+\delta\\
                                                            (1+\|N^{-\alpha}\xni\|-\Upsilon-\delta)\lambdarN_{ij}\pr{\frac{\Upsilon+\delta}{\|N^{-\alpha}\xni\|}\xni}+(\|N^{-\alpha}\xni\|-\Upsilon-\delta)N^{\beta_{ij}}&
                        \text{if }\Upsilon+\delta<\|N^{-\alpha}\xni\|\leq \Upsilon+\delta+1\\
                        N^{\beta_{ij}}&\text{otherwise.}   
                                                           \end{cases}$$
Denote by $\widetilde{z}^N(t)$ the solution to the ODE with modified rate functions and initial condition $\widetilde{z}^N(0)=z^N(0)$. By Assumption \ref{ass:limiting_system}, the functions $N^{-\betar_{ij}}\mlambdarN_{ij}(N^{\alpha}\xni)$ converge uniformly to the functions
  $$\mlambdal_{ij}(\xni)=\begin{cases}
                                                            \lambdal_{ij}(\xni)&\text{if }\|\xni\|\leq \Upsilon+\delta\\
                                                            (1+\|\xni\|-\Upsilon-\delta)\lambdal_{ij}\pr{\frac{\Upsilon+\delta}{\|\xni\|}\xni}+(\|\xni\|-\Upsilon-\delta)&
                        \text{if }\Upsilon+\delta<\|\xni\|\leq \Upsilon+\delta+1\\
                        1&\text{otherwise.}   
                                                           \end{cases}$$
By \eqref{eq:limit_zeta_bounded}, the solution of the limiting reaction network $(\Sp\setminus\In,\Cl,\Rl)$ endowed with the modified kinetics and with initial condition $z(0)$ coincide with $z(t)$ on $[0,T]$. Since for any $y_i\to y_j\in \Rr$, we have
$$\lim_{N\to\infty}\sup_{\xni\in\RR_{\geq0}^{|\Sp\setminus\In|}}\left\|N^{-\alpha}(y_j-y_i)\mlambdarN(N^{\alpha}\xni)-(y^{(i,j)}_j-y^{(i,j)}_i)\mlambdal(\xni)\right\|=0,$$
it follows that
$$\lim_{N\to\infty}\sup_{t\in[0,T]}\|N^{-\alpha}\widetilde{z}^N(t)-z(t)\|=0.$$
This in turn implies that, for $N$ large enough,
$$\sup_{t\in[0,T]}\|N^{-\alpha}\widetilde{z}^N(t)\|\leq \Upsilon+\delta.$$
Therefore for $N$ large enough and for any $t\in[0,T]$, we have $\widetilde{z}^N(t)=z^N(t)$. In particular,
$$\lim_{N\to\infty}\sup_{t\in[0,T]}\|N^{-\alpha}z^N(t)-z(t)\|=0.$$
Moreover by Proposition \ref{prop:convergence},
$$\lim_{N\to\infty}\sup_{t\in[0,T]}\|N^{-\alpha}\xni^N(t)-N^{-\alpha}z^N(t)\|=0.$$
The proof is therefore concluded by the triangular inequality. \qed
\bibliographystyle{plain}
\bibliography{bibliography}

\end{document}